\documentclass[a4paper,11pt,reqno]{amsart}

\usepackage{a4wide}
\makeatletter \@addtoreset{equation}{section}

\makeatother \makeatletter

\usepackage[latin1]{inputenc}
\usepackage[english]{babel}
\usepackage{amsmath,amssymb,amsfonts,amsthm,mathrsfs,upgreek,accents}
\usepackage{url}
\usepackage{mathtools}
\usepackage{color}
\usepackage[colorlinks=true,linkcolor=blue,citecolor=blue,urlcolor=blue,breaklinks]{hyperref}
\usepackage{graphicx,color}
\usepackage{tikz}
\usepackage{pinlabel}

\usepackage{booktabs}
\usepackage{textcomp}
\usepackage{subfigure}

\numberwithin{equation}{section}

\theoremstyle{plain}
\begingroup
\newtheorem{theorem}{Theorem}[section]
\newtheorem{prop}[theorem]{Proposition}

\endgroup

\theoremstyle{definition}
\begingroup

\newtheorem{remark}[theorem]{Remark}
\endgroup

\newcommand{\R}{\mathbb R}
\newcommand{\N}{\mathbb N}

\newcommand{\e}{\varepsilon}

\newcommand{\HH}{{\mathcal H}}

\def\P{\mathcal P}

\def\HH{\mathcal H}

\def\avint{\mathop{\,\rlap{--}\hspace{-.15cm}\int}\nolimits}

\DeclareMathOperator\supp{supp}

\newcommand{\mres}{\mathbin{\vrule height 1.6ex depth 0pt width
0.13ex\vrule height 0.13ex depth 0pt width 1.3ex}}

\title[The spheroid law]{The equilibrium measure for an \\anisotropic nonlocal energy}

\author[J.~A. Carrillo]{J.~A. Carrillo}
\author[J. Mateu]{J. Mateu}
\author[M.~G. Mora]{M.~G. Mora}
\author[L. Rondi]{L. Rondi}
\author[L. Scardia]{L. Scardia}
\author[J. Verdera]{J. Verdera}

\address[J.A. Carrillo]{Department of Mathematics, Imperial College London, United Kingdom}
\email{carrillo@imperial.ac.uk}

\address[J. Mateu]{Department de Matem\`atiques, Universitat Aut\`onoma de Barcelona, Catalonia}
\email{mateu@mat.uab.cat}

\address[M.G. Mora]{Dipartimento di Matematica, Universit\`a di Pavia, Italy}
\email{mariagiovanna.mora@unipv.it}

\address[L. Rondi]{Dipartimento di Matematica, Universit\`a di Milano, Italy}
\email{luca.rondi@unimi.it}

\address[L. Scardia]{Department of Mathematics, Heriot-Watt University, United Kingdom}
\email{L.Scardia@hw.ac.uk}

\address[J. Verdera]{Department de Matem\`atiques, Universitat Aut\`onoma de Barcelona, Catalonia}
\email{jvm@mat.uab.cat}

\begin{document}

\begin{abstract}
In this paper we characterise the minimisers of a one-parameter family of nonlocal and anisotropic energies $I_\alpha$ defined on probability measures in $\R^n$, with $n\geq 3$. The energy $I_\alpha$ consists of a purely nonlocal term of convolution type, whose interaction kernel reduces to the Coulomb potential for $\alpha=0$ and is anisotropic otherwise, and a quadratic confinement. The two-dimensional case arises in the study of defects in metals and has been solved by the authors by means of complex-analysis techniques.
We prove that for $\alpha\in (-1, n-2]$, the minimiser of $I_\alpha$ is unique and is the (normalised) characteristic function of a spheroid. This result is a paradigmatic example of the role of the anisotropy of the kernel on the shape of minimisers. In particular, the phenomenon of loss of dimensionality, observed in dimension $n=2$, does not occur in higher dimension at the value $\alpha=n-2$ corresponding to the sign change of the Fourier transform of the interaction potential. 

\bigskip

\noindent\textbf{AMS 2010 Mathematics Subject Classification:}  31A15 (primary); 49K20 (secondary).

\medskip

\noindent \textbf{Keywords:} nonlocal interactions; potential theory; global minimisers.
\end{abstract}

\maketitle

\begin{section}{Introduction}

There is a vast and multi-disciplinary literature on nonlocal energies, as they are at the crossover of different mathematical fields, and of different applications. Nonlocal energies arise as the macroscopic limit of long-range discrete interactions in the many-particle limit. The expression of the interaction kernel and its properties depend on the particle system of interest: it can model attraction, repulsion or a combination of both; it can be bounded or singular; it can be radial or anisotropic. 

The mathematical literature on nonlocal energies has been mainly focused on the case of radial potentials, which model interactions depending on the mutual distance between the particles only (see, e.g., \cite{BCLR, CCV, CDM, CDP, CHVY, CH, FYK11,SST}). In many applications, however, radial potentials are not realistic, and interactions may depend not only on the inter-particle distance, but also on the angle between their position vector and a given, \textit{preferred} direction (see, e.g., \cite{Carola, CCHS, Kunz, Seidl}). This is for instance the case for many biological systems, e.g., crowds, flocks of birds, schools of fish.

In materials science, some defects in metals, like \textit{dislocations} of edge type, interact via an \textit{anisotropic} potential, and this is the particle system that motivates our work. The anisotropy of the interactions is due to the motion of each dislocation being restricted to a given direction (the Burgers' vector, $\mathbf{b}$, of the dislocation), which is reminiscent of the metal's microscopic lattice structure. Under the simplifying assumption that dislocations are all parallel to each other, and have $\mathbf{b}=\mathbf{e_1}$, they can be modelled as point defects in two dimensions, and their interaction potential is 
$$
W_{\textrm{edge}}(x) = -\log|x| + \frac{x_1^2}{|x|^2}, \quad x\neq 0, \, x=(x_1, x_2)\in\R^2,
$$
see, e.g., \cite{HL}. The anisotropy of the interactions results into anisotropic low-energy dislocation structures (LEDS), dislocation walls in particular. The minimality of vertical one-dimensional structures (walls) was a long-standing conjecture in the engineering literature, and was recently proved in \cite{MRS} for the nonlocal dislocation energy
\begin{equation}\label{e:edge}
I_{\textrm{edge}}(\mu) = \iint_{\mathbb{R}^2\times \mathbb{R}^2}
W_{\textrm{edge}}(x-y) \,d\mu(x) \,d\mu(y) + \int_{\mathbb{R}^2}
|x|^2  \,d\mu(x)
\end{equation}
defined on probability measures $\mu\in \mathcal{P}(\R^2)$ representing the density of defects. For the derivation of an interaction energy related to \eqref{e:edge} (but in a bounded domain) from a semi-discrete strain energy we refer to \cite{MPS}. We also mention the recent work \cite{LL} where, starting from the nonlinear version of the strain-energy model considered in \cite{MPS}, `low-angle grain boundaries' (like vertical walls) are shown to have the optimal energy scaling in accordance with the celebrated Read-Shockley formula.

In this paper we study an $n$-dimensional generalisation of \eqref{e:edge}. More precisely, we consider the family of nonlocal energies
\begin{equation}\label{ce}
I_{\alpha}(\mu) = \iint_{\mathbb{R}^n\times \mathbb{R}^n}
W_{\alpha}(x-y) \,d\mu(x) \,d\mu(y) + \int_{\mathbb{R}^n}
|x|^2  \,d\mu(x)
\end{equation}
defined on $\mu\in \mathcal{P}(\R^n)$, where the interaction potential $W_{\alpha}$ is given by
\begin{equation}\label{V:int0}
W_{\alpha}(x) = W_0(x) + \alpha\frac{x_1^2}{|x|^n}, \qquad  
W_0(x)=
\begin{cases}
\displaystyle-\log|x| \quad & \textrm{if } n=2,\\
\medskip
\displaystyle\frac{1}{|x|^{n-2}} &\textrm{if } n \geq 3,
\end{cases}
\end{equation}
for $x\neq0$, $W_\alpha(0)=+\infty$, with $x=(x_1,\dots, x_n)\in\R^n$, and $\alpha> -1$. Note that $W_{1}=W_{\textrm{edge}}$ for $n=2$. The two-dimensional case was considered in \cite{CMMRSV, MRS} for every $\alpha\in \R$.

The main result of this paper is Theorem~\ref{characterisation}, where we show that, if $n\geq 3$ and $\alpha\in(-1,n-2]$, then 
the functional $I_\alpha$ has a unique minimiser $\mu_\alpha$ in $\P(\R^n)$ and $\mu_\alpha$ is of the form
$$
\mu_\alpha:= \frac{1}{| \Omega_\alpha|} \, \chi_{ \Omega_\alpha}, \quad \Omega_\alpha=  
\left\{x=(x_1,\dots,x_n)\in \R^n: \ \frac{x_1^2}{a(\alpha)^2} + \frac{1}{b(\alpha)^2}\sum_{i=2}^nx_i^2 \leq 1\right\},
$$
for some $a(\alpha), b(\alpha)>0$. In other words, the minimiser of $I_\alpha$ is the (normalised) characteristic function of a $n$-dimensional 
spheroid. Moreover, the spheroid is prolate for $\alpha\in (-1,0)$ (that is, $a(\alpha)>b(\alpha)$) and oblate for $\alpha\in (0,n-2]$ (that is, $a(\alpha)<b(\alpha)$).

\subsection{Our approach and discussion.}
In most of the cases treated in the mathematical literature on nonlocal systems, the interaction kernel is assumed to be radial, and 
one of the goals is to show that the corresponding minimiser is radially symmetric (or to show that the minimiser is unique, which trivially implies its radial symmetry), see e.g. \cite{BCLR, CCV, CDM, CDP, CHVY, CH, FYK11,SST}. Radial symmetry is paramount in the identification of the minimiser in the classical case of purely Coulomb interactions, corresponding in our setting to $\alpha=0$ (see \cite{Fro} and \cite{SaTo} for $n=2$, and  \cite{CGZ} for $n\geq 3$). Explicity characterising the minimiser, or even understanding its shape and general properties, is therefore much more challenging in the case of anisotropic interactions. To the best of our knowledge, this has been previously done only in \cite{MRS, CMMRSV}, in dimension $n=2$.

Our result generalises to any dimension $n\geq 3$ the work \cite{CMMRSV} and is another paradigmatic example of the role of the anisotropy of the kernel on the shape of minimisers. 

A key step in \cite{CMMRSV} was to compute exactly the (gradient of the) potential $W_\alpha\ast \mu_{a,b}$, with $\mu_{a,b}$ being the (normalised) characteristic function of an ellipse of semi-axes $a$ and $b$, and to impose the Euler-Lagrange conditions associated to $I_\alpha$. 
Also for $n\geq 3$ we prove the minimality of spheroids via the Euler-Lagrange conditions (see \eqref{EL-1}--\eqref{EL-2}), in the range $\alpha\in (-1,n-2]$ for which $I_\alpha$ is both well-defined and strictly convex (hence, the Euler-Lagrange conditions are necessary and sufficient for minimality). To do so, we need to compute the potential $W_\alpha\ast \mu_{a,b}$, with $\mu_{a,b}$ being the (normalised) characteristic function of a spheroid $\Omega(a,b)$ of semi-axis $a$ in the $x_1$-direction and $b$ in all the other directions. The computation of $W_\alpha\ast \mu_{a,b}$ for $n\geq 3$ is substantially different from the two-dimensional case in \cite{CMMRSV}. For $n=2$ it was crucial to rewrite the potential in complex variables and recognise that $\nabla W_0\ast \mu_{a,b}$ is the Cauchy transform of the ellipse, which had been computed for instance in \cite{Joans} for rotating vortex patches in fluid dynamics. Then the gradient of the anisotropic part of the potential in complex coordinates was computed by noting that it could be written as a suitable complex derivative of the fundamental solution of the operator $\partial^2$, where $\partial = \partial/\partial z$.

Such complex-analysis techniques are clearly not available in the higher-dimensional case, which we tackle here by means of the following strategy. We write $\Phi_\alpha:=W_\alpha\ast \mu_{a,b}$, and $\alpha\Psi:= \Phi_\alpha-\Phi_0$.  First of all, the expression of the Coulomb potential $\Phi_0$ of a spheroid in $\R^n$ is well-known  (see, e.g., \cite{Difratta, Kellog}). The challenge is to express the anisotropic potential $\Psi$ in terms of the known potential $\Phi_0$. We do it differently in $\Omega(a,b)$ and outside $\Omega(a,b)$. In $\Omega(a,b)$ we show that $\Psi$ can be obtained by differentiating $\Phi_0$ with respect to the aspect ratio $a^2/b^2$ of the spheroid. In $\R^n\setminus \Omega(a,b)$, instead, $\Psi$ is obtained by differentiating $\Phi_0$ with respect to a parameter spanning a family of spheroids confocal with $\Omega(a,b)$, using the fact that the expression of $\Phi_0$ is invariant on confocal spheroids (see \eqref{derivative=0-n}). 

With the expression of $\Phi_\alpha$ at hand, we then impose the Euler-Lagrange conditions \eqref{EL-1}--\eqref{EL-2}. We find that the first condition is satisfied, for $\alpha\in (-1,n-2]$, by at least a pair $(a(\alpha),b(\alpha))$ of semi-axes, with $a(\alpha)>b(\alpha)>0$ for $\alpha\in (-1,0)$ and $0<a(\alpha)<b(\alpha)$ for $\alpha\in (0,n-2]$. Hence there is at least one stationary, non-degenerate spheroid $\Omega(a(\alpha),b(\alpha))$ for the energy $I_\alpha$. We then show that, for any spheroid $\Omega(a(\alpha),b(\alpha))$ for which the stationarity condition \eqref{EL-1} is satisfied, also the 
unilateral condition \eqref{EL-2} is satisfied. Since \eqref{EL-1}--\eqref{EL-2} are necessary and sufficient conditions for minimality, this implies that any spheroid $\Omega(a(\alpha),b(\alpha))$ satisfying \eqref{EL-1} is in fact a minimiser for $I_\alpha$. The strict convexity of the energy then gives uniqueness of the minimiser, and in particular implies that there is only one spheroid satisfying~\eqref{EL-1}. This approach can be carried out also in the two-dimensional case (see \cite{Scag}).    

\subsubsection{Dimensionality of minimisers for $n=2$ and $n\geq 3$.} For the energy \eqref{e:edge}, it was shown in \cite{MRS} that the unique minimiser is one dimensional, and is given by the semi-circle law on the vertical axis,
$$
\mu_{\textrm{edge}} = \frac{1}{\pi}\delta_0\otimes \sqrt{2-x_2^2} \, \mathcal{H}^1\mres(-\sqrt2,\sqrt2).
$$
The semi-circle law also arises as the unique minimiser of the one-dimensional logarithmic energy with quadratic confinement (see \cite{Wi}), and represents in that case the optimal positions of the eigenvalues of a Hermitian random matrix with Gaussian entries. 

We recall that the minimiser of the Coulomb-gas energy $I_0$ for $n=2$ is the two-dimensional measure $\mu_0 = \frac1\pi \chi_{B_1(0)}$, the so-called circle-law, also well-known in the context of random matrices. In fact, the minimiser of $I_0$ is the normalised characteristic function of a ball in any dimension. The change of dimension of the minimiser of the energy $I_\alpha$, for $n=2$, between $\alpha=0$ and $\alpha=1$ was investigated by the authors in \cite{CMMRSV}. In \cite{CMMRSV} it was shown in particular that the minimiser of $I_\alpha$, for $n=2$ and $\alpha\in (-1,1)$, is the two-dimensional measure 
$$
\mu_\alpha = \frac1\pi \frac{1}{\sqrt{1-\alpha^2}} \chi_{\Omega(\sqrt{1-\alpha},\sqrt{1+\alpha})}, 
$$
with
\begin{equation*}
\Omega(\sqrt{1-\alpha},\sqrt{1+\alpha})=\left\{x=(x_1,x_2)\in\R^2:\
\frac{x_1^2}{1-\alpha}+\frac{x_2^2}{1+\alpha}<1\right\}.
\end{equation*}
Hence the minimiser of $I_\alpha$ has full dimension for $\alpha\in (-1,1)$, and is one-dimensional for both $\alpha\leq -1$ and $\alpha\geq 1$, being respectively the semi-circle law on the horizontal or the vertical axis. 

A question left open in \cite{CMMRSV} was to understand why there is a change of dimension of the minimiser $\mu_\alpha$ at $\alpha=\pm1$, for $n=2$. The relation between the dimensionality of the minimiser of a nonlocal energy and the singularity of the interaction kernel is a fascinating and subtle problem. The available results in the literature are usually of the form of a lower bound for the dimension of the measure (see, e.g., \cite{BCLR}), which is helpful if the goal is to prove that the dimension is full, but less so to prove that there is a loss of dimension.

For the energy $I_\alpha$ in dimension $n=2$, since the Fourier transform of $W_\alpha$ changes sign exactly at the values $\alpha=\pm1$, it was natural to conjecture that the change of dimension could be due to the change of sign of $\widehat{W}_\alpha$. Similarly, since 
the Laplacian of $W_\alpha$ is
\begin{align*}
\Delta W_\alpha(x) =   -(1-\alpha)\partial_{x_1}^2 \log|x| - (1+\alpha)\partial_{x_2}^2 \log|x|,
\end{align*}
it was reasonable to expect that the singular behaviour exhibited by $\mu_\alpha$ at $\alpha=\pm1$ was a consequence of the `degeneracy' of $\Delta W_\alpha$ at those values.

The analysis done in this paper demonstrates that the situation is more delicate. While it is still plausible to expect that a positive Fourier transform (or a non-degenerate Laplacian) results into a fully-dimensional minimiser, the contrary is not true, at least for $n\geq 3$. Indeed while for $n\geq 3$ the Fourier transform $\widehat{W}_\alpha$ of the interaction kernel changes sign at $\alpha=n-2$ (see \eqref{FT3d}), and similarly
$$
\Delta W_\alpha(x) =\bigg(1-\frac{\alpha}{n-2}\bigg) \partial_{x_1}^2\bigg(\frac1{|x|^{n-2}}\bigg) + \bigg(1+\frac{\alpha}{n-2}\bigg) \sum_{i=2}^n\partial_{x_i}^2\bigg(\frac1{|x|^{n-2}}\bigg),
$$
we prove that the minimiser of $I_{\alpha}$ is the characteristic function of a non-degenerate spheroid also for the limit value $\alpha=n-2$.

\subsubsection{The shape of minimisers for $\alpha>0$ and $\alpha<0$.}
In the two-dimensional case $n=2$, changing sign to $\alpha$ corresponds to swapping $x_1$ and $x_2$ (up to a constant in the energy), 
due to the zero-homogeneity of the energy. Hence it is sufficient to characterise the minimisers of $I_\alpha$ for $\alpha>0$, which is what we did in \cite{CMMRSV}.

This is no longer true for $n\geq 3$, since in this case there is only one privileged coordinate.  
Intuitively, configurations elongated on $x_1$ are penalised for $\alpha>0$, and preferred for $\alpha<0$ (see, e.g., \eqref{rewrite:W}), hence the minimisers for $\alpha>0$ and $\alpha<0$ 
cannot be congruent up to a rotation, which is the case in dimension two.
More precisely, for $n\geq3$ we may write
$$
W_{\alpha}(x) = (1+\alpha) W_0(x) -\alpha \frac{1}{|x|^n} \sum_{i=2}^n x_i^2. 
$$
Thus, changing sign to $\alpha$ corresponds not only to a change in the anisotropy, but also to a rescaling of the Coulomb kernel. 
This also suggests that a different behaviour of the energy should be expected at $\alpha=-1$ in dimension $n=2$ and $n\geq3$, as discussed next.

\subsubsection{The limiting case $\alpha=-1$} In dimension $n\geq 3$, for $\alpha=-1$ the anisotropy cancels completely the $x_1$-component of the Coulomb potential, since from \eqref{V:int0}
$$
W_{-1}(x) = \frac{1}{|x|^{n-2}} - \frac{x_1^2}{|x|^n} = \frac{x_2^2+\ldots+x_n^2}{|x|^n}, \quad x\neq 0.
$$ 
As a consequence, there is a discrepancy with the situation for $\alpha\in (-1,n-2]$: While $I_\alpha$ is lower semicontinuous for any $\alpha\in (-1,n-2]$, the functional $I_{-1}$ with the kernel $W_{-1}$ above (and $W_{-1}(0)=+\infty$) is not lower semicontinuous (see Remark~\ref{alfaneg}). In particular, $I_{-1}$ does not describe the asymptotic behaviour
of the functionals $I_\alpha$, as $\alpha\to-1^+$ (see Remark~\ref{asymrem}).
We resolve this issue in Section~\ref{Section-1}, where we characterise the $\Gamma$-limit $J_\ast$ of $I_\alpha$, as $\alpha\to-1^+$,
in Fourier space, on probability measures with compact support. This partial representation allows us to show strict convexity of $J_\ast$ on a class of measures that is the relevant one for minimisation, and hence to deduce uniqueness of the minimiser for $J_\ast$. Moreover, we show that the minimiser is, also in this limiting case, the (normalised) characteristic function of an $n$-dimensional (prolate) spheroid.

\subsection{Open questions and future work.}
There are several questions that we will address in future work. 
We believe that ellipses, or spheroids, arise as minimisers of more general anisotropic energies. A first step would be to consider interaction kernels of the form $W_0 + \alpha W_{\textrm{aniso}}$, with 
$$
W_{\textrm{aniso}}(x) = \frac{x_1^{2\kappa}}{|x|^{n-2+2\kappa}}, \quad \kappa\in \N.
$$
It is plausible to expect that in the two-dimensional case $n=2$ and in a suitable range of $\alpha$ the minimisers are ellipses. It would be interesting to understand whether, as for $\kappa=1$, they shrink to a segment for some special value $\alpha^\ast$ within, or at the boundary of, the interval of strict convexity of the corresponding energy (or of positivity of the Fourier transform of $W_0 + \alpha W_{\textrm{aniso}}$). This analysis would help to shed more light on what causes the loss of dimension of the minimising measure.

\end{section}


\begin{section}{Existence and uniqueness of the minimiser of $I_{\alpha}$ for $\alpha\in (-1,n-2]$, $n\geq 3$}\label{preliminary-section}

In this section we prove that for every $\alpha\in (-1,n-2]$ the nonlocal
energy $I_{\alpha}$ defined in \eqref{ce}, for $n\geq 3$, has a unique minimiser
$\mu_\alpha\in \mathcal{P}(\mathbb{R}^n)$, and that the minimiser
has a compact support.


\begin{prop}\label{exist+uniq}
Let $n\geq 3$, and let $\alpha\in (-1,n-2]$. Then the energy $I_{\alpha}$ is well defined
on $\mathcal{P}(\mathbb{R}^n)$, is strictly convex on the class of
measures with compact support and finite interaction energy, and has
a unique minimiser in $\mathcal{P}(\mathbb{R}^n)$. Moreover, the
minimiser has compact support and finite energy.
\end{prop}

\begin{proof}
The case $n =2$ has been proved in \cite[Section~2]{MRS} and \cite[Proposition 2.1]{CMMRSV}. For $n\geq 3$ the proof follows by a similar argument. For the convenience of the reader we outline
the main steps of the proof.\smallskip

\noindent \textit{Step~1: Well definiteness of $I_\alpha$.} 
Since $\alpha> -1$, if we write $W_\alpha$, for $x\neq 0$, as
\begin{equation}\label{rewrite:W}
W_{\alpha}(x) = 
\frac{1}{|x|^n}\bigg((1+\alpha) x_1^2 + \sum_{i=2}^n x_i^2\bigg), 
\end{equation}
we can immediately see that the energy is well-defined and non-negative on $\mathcal{P}(\mathbb{R}^n)$.\smallskip

\noindent \textit{Step~2: Existence of a compactly supported
minimiser.} 
First of all, it is easy to see that $I_\alpha(\mu_B)<+\infty$, where $\mu_B= \frac{1}{|B_1(0)|}
\chi_{B_1(0)}$. This implies that $\inf_{\mathcal{P}(\mathbb{R}^n)}
I_\alpha < +\infty$. Moreover, we have that
\begin{align}
W_{\alpha}(x-y) + \frac12(|x|^2+|y|^2) 
\geq\frac12(|x|^2+|y|^2). \label{bound:compact}
\end{align}
This lower bound provides tightness and hence
compactness with respect to narrow convergence for minimising
sequences, that, together with the lower semicontinuity of
$I_\alpha$, guarantees the existence of a minimiser.
As in \cite[Section~2.2]{MRS}, one can show that any minimiser of
$I_{\alpha}$ has compact support, again by
\eqref{bound:compact}.\smallskip

\noindent \textit{Step~3: Strict convexity of $I_\alpha$ and
uniqueness of the minimiser.} We prove that
\begin{equation}\label{cvx0}
\int_{\R^n} W_{\alpha}\ast (\nu_1-\nu_2) \,d(\nu_1-\nu_2) > 0
\end{equation}
for every $\nu_1, \nu_2 \in \mathcal{P}(\R^n)$, $\nu_1\neq \nu_2$,
with compact support and finite interaction energy, namely such that
$\int_{\R^n} (W_{\alpha}\ast\nu_i) \, d\nu_i < + \infty$ for
$i=1,2$. Condition \eqref{cvx0} implies strict convexity of
$I_{\alpha}$ on the set of probability measures with compact support
and finite interaction energy and, consequently, uniqueness of the
minimiser.

To prove \eqref{cvx0}, we follow again the same strategy as in \cite[Section~2.3]{MRS}.
The idea consists in rewriting the interaction energy of
$\nu:=\nu_1-\nu_2$ in Fourier space, as
\begin{equation}\label{Fourier:space}
\int_{\R^n} W_{\alpha}\ast \nu \,d \nu = \int_{\R^n}
\widehat{W}_\alpha(\xi) |\widehat\nu(\xi)|^2\, d\xi,
\end{equation}
and proving that $\widehat{W}_\alpha$ is a positive distribution.

\noindent \textit{Step~3.1: Computation of $\widehat{W}_\alpha$.} Note that $W_{\alpha}$ is a tempered distribution, namely
$W_{\alpha}\in{\mathcal S}'$, where $\mathcal S$ denotes the
Schwartz space; hence also $\widehat W_{\alpha}\in{\mathcal S}'$. We recall
that $\widehat W_{\alpha}$ is defined by the formula
$$
\langle \widehat W_{\alpha}, \varphi\rangle := \langle W_{\alpha}, \widehat
\varphi\,\rangle \qquad \text{ for every } \varphi\in{\mathcal S}
$$
where, for $\xi\in\R^n$,
\begin{equation*}
\widehat \varphi(\xi):=\int_{\R^n}\varphi(x)e^{-2\pi i\xi\cdot x}\, dx.
\end{equation*}
To compute $\widehat{W}_\alpha$ it is convenient to rewrite $W_\alpha$ as 
\begin{equation*}
W_\alpha(x) = \Big(1+\frac{\alpha}n\Big) \frac{1}{|x|^{n-2}} + \frac{\alpha}n\,\frac{1}{|x|^n} \bigg((n-1) x_1^2 -\sum_{i=2}^n x_i^2\bigg),
\end{equation*}
namely as the sum of the $n$-dimensional Coulomb potential (up to a multiplicative constant) and the ratio between a homogeneous harmonic polynomial of degree two and a power of $|x|$. 
By \cite[eq.(32), p.73]{Stein} and by \cite[Exercise 1, p.154]{Folland} we have  that the Fourier transform $\widehat W_{\alpha}$ of $W_{\alpha}$ is given by
\begin{align}\label{FT3d}
\langle \widehat W_{\alpha}, \varphi\rangle
 &=  \Big(1+\frac{\alpha}n\Big)\frac{(n-2)\pi^{\frac n2-2}}{2\Gamma(\frac n2)} \int_{\R^n}\frac{1}{|\xi|^2}\varphi(\xi) d\xi - \frac{\alpha}{n} \,\frac{\pi^{\frac n2-2}}{\Gamma(\frac n2)} \int_{\R^n} \frac{(n-1) \xi_1^2 - \sum_{i=2}^n\xi_i^2}{|\xi|^4} \, \varphi(\xi) d\xi \nonumber\\
&= \frac{\pi^{\frac n2-2}}{2\Gamma(\frac n2)}  \int_{\R^n} \frac{(n-2-\alpha) \xi_1^2 + (n-2+\alpha) \sum_{i=2}^n \xi_i^2}{|\xi|^4}   \, \varphi(\xi) d\xi
\end{align}
for every $\varphi\in{\mathcal S}$.
Thus, \eqref{FT3d} implies that $\widehat W_{\alpha}$ is a positive function in $L^1_{\textrm{loc}}(\R^n)$ for every $\alpha\in (-1, n-2]$,
hence in particular a positive tempered distribution.\smallskip

\noindent \textit{Step~3.2: Proof of \eqref{Fourier:space}.} 
%
We start by proving that \eqref{Fourier:space} holds when $\nu$ is a non-negative finite Borel measure with compact support, 
where we understand that the two sides of the formula are either both finite and coincide, or both equal to $+\infty$. 

We proceed by regularisation. Let $\varphi\in C^\infty_c(B_1(0))$ be non-negative, radial, and with $\int_{\R^n} \varphi (x)\,dx =1.$ For $\e>0$ we define
\begin{equation}\label{convol}
\varphi_\e(x):= \frac{1}{\e^n} \varphi\left(\frac{x}{\e}\right) \qquad  \text{ and} \qquad \nu_\e: = \nu \ast \varphi_\e.
\end{equation}
We claim that
\begin{equation}\label{parep}
\int_{\R^n} (W_\alpha\ast \nu_\e)(x) \nu_\e(x) \,dx = \int_{\R^n} \widehat{W}_\alpha(\xi) |\widehat{\nu_\e}(\xi)|^2 \, d\xi.
\end{equation}
To show this, let us set $f:=W_\alpha \ast \nu_\e$ and $g:=\nu_\e$, and note that $g \in C^{\infty}_c(\R^n)$ and $f \in C^{\infty}(\R^n)$.
Moreover, since $\widehat{\nu_\e}\in{\mathcal S}$ and $\widehat{W}_\alpha\in L^1_{\textrm{loc}}(\R^n)$
behaves as $1/|\xi|^2$ at infinity by \eqref{FT3d}, we have that $\widehat{f}=\widehat{W}_\alpha\,\widehat{\nu_\e}\in L^1(\R^n)$.
Let $\psi \in C^{\infty}_c(\R^n)$ be such that $\psi = 1$ on $B_1(0)$ and 
let $R>0$ be such that the support of $g$ is contained in $B_R(0)$. If $\tau>0$ is such that $\tau R < 1$, then,
by Parseval formula,
\begin{eqnarray}
\int_{\R^n} f(x)g(x)\,dx & = & \int_{\R^n} \psi(\tau x)f(x)g(x) \,dx 
\nonumber
\\
& = & \int_{\R^n} (\widehat{\psi(\tau\, \cdot)} \ast \widehat{f}\, )(\xi) 
\,\overline{\widehat{g}(\xi)}\, d\xi 
= \int_{\R^n} (\widehat{\psi}_\tau \ast \widehat{f}\,)(\xi)\, \overline{\widehat{g}(\xi)}\, d\xi, 
\label{lime}
\end{eqnarray}
where $\widehat{\psi}_\tau(x):= \tau^{-n} \widehat{\psi}(x/\tau)$. Using that $\widehat\psi\in\mathcal S\subset L^1(\R^n)$, $\int_{\R^n}\widehat\psi(\xi)\,d\xi=\psi(0)=1$, and that $\widehat f\in L^1(\R^n)$, it is easy to see that $\widehat{\psi}_\tau \ast \widehat{f}$ converges to $\widehat f$
in $L^1(\R^n)$, as $\tau\to0$. 
In fact,
$$
\|\widehat{\psi}_\tau \ast \widehat{f} -\widehat{f} \|_{L^1}\leq \int_{\R^n}|\widehat\psi(\eta)|\int_{\R^n} |\widehat{f}(\xi-\tau\eta)-\widehat{f}(\xi)|\,d\xi d\eta
$$
and one can conclude by the continuity of translations in the $L^1$-norm and by the Dominated Convergence Theorem.
Since $\widehat{g}\in L^{\infty}(\R^n)$, we deduce that
$$
\lim_{\tau\to0} \int_{\R^n} (\widehat{\psi}_\tau \ast \widehat{f}\,)(\xi)\, \overline{\widehat{g}(\xi)}\, d\xi
= \int_{\R^n} \widehat{f}(\xi)\, \overline{\widehat{g}(\xi)}\, d\xi,
$$ 
which, together with \eqref{lime}, proves \eqref{parep}.

We now let $\e \rightarrow 0$ in \eqref{parep}. For the right-hand side we observe that
for every~$\xi\in\R^n$
$$
\widehat{\varphi_\e}(\xi) = \widehat{\varphi}(\e \xi) \to \widehat{\varphi}(0) =1,
$$
as $\e\to0$, and that
$\|\widehat{\varphi_\e}\|_{L^\infty}\le \|\varphi \|_{L^1} =1$
for every $\e >0$.
Therefore, either by the Dominated Convergence Theorem or the Fatou Lemma, we have
\begin{equation}\label{trickJV}
\int_{\R^n}  \widehat{W}_\alpha(\xi) |\widehat{\nu_\e}(\xi)|^2 \, d\xi = \int_{\R^n}  \widehat{W}_\alpha(\xi) |\widehat{\nu}(\xi)|^2 
|\widehat{\varphi_\e}(\xi) |^2 \, d\xi \ \to \ \int_{\R^n}  \widehat{W}_\alpha(\xi) |\widehat{\nu}(\xi)|^2  \, d\xi,
\end{equation}
as $\e\to0$, even if the right-hand side is infinite.

To deal with the left-hand side of \eqref{parep} we note that, since $\alpha>-1$, 
there exists a positive constant $C=C(\alpha)$ such that
\begin{equation}\label{dobliu}
 \frac{1}{C}\, W_0(x) \le W_\alpha(x) \le C\, W_0(x) \qquad \text{ for every } x\in\R^n.
\end{equation}
Hence,
\begin{equation*}
(W_\alpha\ast \varphi_\e)(z) \le C(W_0\ast \varphi_\e)(z) \qquad \text{ for every } z\in\R^n.
\end{equation*}
Since $W_0$ is superharmonic and $\varphi$ is radial with integral $1$,
the mean value property on spheres yields
\begin{equation*}
(W_0\ast \varphi_\e)(z) 
\le W_0(z) \qquad \text{ for every } z\in\R^n.
\end{equation*}
Thus, combining the three previous inequalities,
\begin{equation}\label{dobliufun}
(W_\alpha\ast \varphi_\e)(z) \le C\, W_0(z) \le C^2\,  W_\alpha(z) \qquad \text{ for every } z\in\R^n.
\end{equation}
Note that for every $z \in \R^n$
\begin{equation}\label{kereg}
(W_\alpha\ast \varphi_\e)(z) \to W_\alpha(z), 
\end{equation}
as $\e\to0$, since $W_\alpha$ is continuous as a function with values into $[0,+\infty]$. 
Owing to the convergence \eqref{kereg} and the domination \eqref{dobliufun},
we can either apply the Dominated Convergence Theorem or the Fatou Lemma to deduce that
\begin{equation}\label{conv}
\iint_{\mathbb{R}^n\times \mathbb{R}^n}(W_\alpha\ast \varphi_\e)(x-y) \, d\nu(x)\,d\nu(y) \ \to \ 
\iint_{\mathbb{R}^n\times \mathbb{R}^n} W_\alpha(x-y)  \, d\nu(x)\, d\nu(y),
\end{equation}
as $\e\to0$, even if the right-hand side is infinite.

We now go back to the left-hand side of \eqref{parep} and observe that
\begin{equation*}
\int_{\R^n} (W_\alpha\ast \nu_\e)(x)\nu_\e(x) \, dx = \iint_{\mathbb{R}^n\times \mathbb{R}^n}(W_\alpha\ast \varphi_\e\ast \varphi_\e)(x-y) \, d\nu(x)\,d\nu(y).
\end{equation*}
Note that $(\varphi_\e \ast \varphi_\e)(x) = \e^{-n}(\varphi \ast \varphi)(x/\e)$ and that $\varphi \ast \varphi$ inherits the properties of $\varphi$: it is radial, belongs to $C^\infty_c(\R^n)$, and $\int_{\R^n} (\varphi \ast \varphi)(x)\,dx =1$. Therefore, \eqref{conv} holds with $\varphi_\e$ replaced by $\varphi_\e \ast \varphi_\e$. This concludes the proof of \eqref{Fourier:space} for a non-negative measure $\nu$.

\smallskip

We now prove \eqref{Fourier:space} for a signed and neutral measure $\nu:=\nu_1-\nu_2$, where $\nu_1, \nu_2 \in \mathcal{P}(\R^n)$ have compact support and finite interaction energy. First of all, by using \eqref{Fourier:space} for $\nu_1+\nu_2$ we have that
\begin{eqnarray*}
\int_{\R^n}  (W_\alpha \ast (\nu_1+\nu_2))(x)\,d(\nu_1+\nu_2)(x) & = & \int_{\R^n}  \widehat{W}_\alpha(\xi) |\widehat{\nu}_1(\xi)+\widehat{\nu}_2(\xi)|^2\,d\xi 
\\
& \le & 2 \int_{\R^n}  \widehat{W}_\alpha(\xi)(|\widehat{\nu}_1(\xi)|^2 +|\widehat{\nu}_2(\xi)|^2 )\, d\xi <  +\infty.
\end{eqnarray*}
By expanding both sides of the identity above and using \eqref{Fourier:space} for $\nu_1$ and $\nu_2$ we get
$$
\int_{\R^n}  (W_\alpha\ast \nu_1)(x)\,d\nu_2(x)= \int_{\R^n}  \widehat{W}_\alpha(\xi)\,{\rm Re}\big(\widehat{\nu}_1(\xi)\overline{\widehat{\nu}_2(\xi)}\,\big) \,d\xi,
$$
which, by using again \eqref{Fourier:space} for $\nu_1$ and $\nu_2$, gives
$$
\int_{\R^n}  (W_\alpha\ast (\nu_1-\nu_2))(x)\,d(\nu_1-\nu_2)(x) = \int_{\R^n} \widehat{W}_\alpha(\xi)|\widehat{\nu}_1(\xi)-\widehat{\nu}_2(\xi)|^2\,d\xi.
$$
Since the right-hand side is strictly positive for $\nu_1\neq\nu_2$, equation \eqref{cvx0} is proved.
\end{proof}

\begin{remark}\label{alfaneg} 
In the case $n\geq3$, unlike in the two-dimensional case, the energy is not well-defined for $\alpha<-1$. Indeed, writing $W_\alpha$ as in \eqref{rewrite:W}, we can see that the two terms in the right-hand side of \eqref{rewrite:W} are both unbounded for $|x|$ close to zero, and have opposite sign. 

Even for $\alpha=-1$ the situation is subtle. The functional $I_{-1}$ with kernel $W_{-1}$ defined as in \eqref{V:int0} is not lower semicontinuous with respect to narrow convergence: Indeed, the probability measures $\mu_k:=k\HH^1\mres \big((0,\frac1k)\times\{0\}^{n-1}\big)$ converge narrowly to the Dirac delta at $0$, but 
\begin{equation}\label{not-ls-mg}
0=\lim_{k\to +\infty} I_{-1}(\mu_k) < I_{-1}(\delta_0) = +\infty.
\end{equation}
So, in particular, $I_{-1}$ is not the $\Gamma$-limit of $I_\alpha$ for $\alpha\to -1^{+}$. Moreover, \eqref{not-ls-mg} implies that the relaxed functional $\overline{I_{-1}}$
of $I_{-1}$ is equal to $0$ at $\delta_0$, which is therefore a minimiser of $\overline{I_{-1}}$.
Note that the kernel $W_{-1}$ is also not lower semicontinuous, since $W_{-1}(0)=+\infty$ and 
$$
\liminf_{x\to0} W_{-1}(x)=0.
$$
If we define a new kernel $\widetilde W_{-1}$ to be as $W_{-1}$ for $x\neq 0$, and $\widetilde W_{-1}(0):=0$, then $\widetilde W_{-1}$ is lower semicontinuous, and the corresponding functional $\widetilde{I}_{-1}$ has a  unique minimiser, which is simply the Dirac delta at $0$. 

The case $\alpha=-1$ will be discussed in detail in Section \ref{Section-1}.
\end{remark}

\end{section}


\begin{section}{Minimality of spheroids}\label{proof-section}

It is a standard computation in potential theory to show that any minimiser $\mu$ of $I_\alpha$ must satisfy
the following Euler-Lagrange conditions: There exists $C\in\R$ such
that
\begin{align}\label{EL-1}
&(W_{\alpha}\ast \mu)(x) + \frac{|x|^2}2 = C \quad  \text{for
$\mu$-a.e.\ } x\in \supp\mu,
\\ \label{EL-2}
&(W_{\alpha}\ast \mu)(x) + \frac{|x|^2}2 \geq C \quad \text{for
q.e.\ }x\in \R^n,
\end{align}
where {\em quasi everywhere} (q.e.) means up to sets of zero
capacity (see \cite[Chapter~I, Theorem~1.3]{SaTo} or \cite[Section~3.1]{MRS}).  The Euler-Lagrange conditions
\eqref{EL-1}--\eqref{EL-2} are in fact equivalent to
minimality for $\alpha\in (-1,n-2]$ due to
Proposition~\ref{exist+uniq}. We refer to \cite[Section 3.1]{MRS}  for
details.

Let $a,b>0$, and let $\Omega(a,b)\subset \R^n$ denote the ellipsoid with semi-axis $a$ in the $x_1$ direction and the other semi-axes of the same length $b$, namely
\begin{equation*}
\Omega(a,b):= \left\{x=(x_1,\dots,x_n)\in \R^n: \ \frac{x_1^2}{a^2} + \frac{1}{b^2}\sum_{i=2}^nx_i^2 \leq 1\right\}.
\end{equation*}
This special ellipsoid is called \textit{oblate spheroid} if $a<b$ and \textit{prolate spheroid} if $a>b$.

The main result is the following. 

\begin{theorem}\label{characterisation}
Let $n\geq 3$ and $\alpha\in(-1,n-2]$. There exist $a(\alpha), b(\alpha)>0$ such that the measure 
$$
\mu_\alpha:= \frac{1}{| \Omega_\alpha|} \, \chi_{ \Omega_\alpha}, \qquad \Omega_\alpha:=  \Omega(a(\alpha), b(\alpha)),
$$
is the unique minimiser of the functional $I_\alpha$ in $\mathcal P(\R^n)$ and satisfies the Euler-Lagrange conditions
\begin{align}\label{EL-1-real}
(W_{\alpha}\ast \mu_\alpha)(x) + \frac{|x|^2}2 &= C_\alpha \quad  
\text{for every\ } x\in \Omega_\alpha,\\\label{EL-2-real}
(W_{\alpha}\ast \mu_\alpha)(x) + \frac{|x|^2}2 &\geq C_\alpha \quad \text{for every\ }x\in \R^n,
\end{align}
with $C_\alpha= 2I_\alpha(\mu_\alpha)- \frac12\int_{\mathbb{R}^n} |x|^2 \,d\mu_\alpha(x)$. Moreover, the spheroid $\Omega_\alpha$ is prolate for $\alpha\in (-1,0)$ and oblate for $\alpha\in (0,n-2]$.
\end{theorem}

\begin{remark}
For $\alpha=0$, $n\geq3$, it is well-known that the unique minimiser of  the Coulomb energy $I_0$ is the normalised characteristic function of the ball centred at $0$ with radius $(n-2)^{1/n}$ (see, e.g., \cite[Corollary~1.3]{CGZ}). In other words, $a(0)=b(0)=(n-2)^{1/n}$. In the proof below we will focus only on the case $\alpha\neq0$.
\end{remark}

We split the proof of Theorem \ref{characterisation} into Section~\ref{sect:EL1}, where we prove the stationarity condition \eqref{EL-1-real}, and Section~\ref{sect:EL2-n}, where we prove \eqref{EL-2-real}. The heart of the proof consists in the exact evaluation of the convolution 
$$
\Phi_\alpha:= W_\alpha\ast \frac{1}{|\Omega(a,b)|}\chi_{\Omega(a,b)}
$$
both in $\Omega(a,b)$ (in Section~\ref{sect:EL1}) and in $\R^n\setminus \Omega(a,b)$ (in Section~\ref{sect:EL2-n}), for $a,b>0$.

We write 
\begin{equation}\label{defPsi}
\Phi_\alpha(x)= \Phi_0(x) + \alpha \Psi(x), \quad \textrm{with } \, \Psi(x):= \avint_{\!\!\Omega(a,b)}\frac{(x_1-y_1)^2}{|x-y|^n} \,dy. 
\end{equation}

\subsection{The condition \eqref{EL-1-real} on spheroids}\label{sect:EL1} Let $\alpha\in (-1,n-2]$. We claim that there exist $a(\alpha), b(\alpha)>0$ (with $a(\alpha)< b(\alpha)$ for $\alpha\in (0,n-2]$ and  $a(\alpha)> b(\alpha)$ for $\alpha\in (-1,0)$) such that
\begin{align}
(W_{\alpha}\ast \mu_\alpha)(x) + \frac{|x|^2}2 = C_\alpha \quad  
\text{for every\ } x\in \Omega_\alpha, \qquad \mu_\alpha:= \frac{1}{| \Omega_\alpha|} \, \chi_{ \Omega_\alpha},
\end{align}
where we recall that $\Omega_\alpha=  \Omega(a(\alpha), b(\alpha))$. 

In the two-dimensional case studied in \cite{CMMRSV}, we computed the semi-axes of $\Omega_\alpha$ in terms of $\alpha$, and deduced the explicit values $a=\sqrt{1-\alpha}$ and $b=\sqrt{1+\alpha}$. For $n\geq  3$ we do not have explicit expressions for the semi-axes in terms of $\alpha$.

\subsubsection{The potential inside a spheroid}\label{sect:inside} 
In this section, for $\alpha\in (-1,n-2]$ and $a,b>0$, we evaluate $\Phi_\alpha(x)$ with $x\in \Omega(a,b)$. We start by recalling the case $\alpha=0$ of the Coulomb potential, namely 
$$
\Phi_0(x)= \avint_{\!\!\Omega(a,b)}\frac{1}{|x-y|^{n-2}} \,dy.
$$
For $x\in \Omega(a,b)$ we have that 
\begin{align}\label{Riesz:ellipsoid-n}
\Phi_0(x)=& 
 \frac{n(n-2)}4\int_0^\infty \left(1-\frac{x_1^2}{a^2+s} - \frac{r^2}{b^2+s}\right) \frac{ds}{\sqrt{a^2+s}(b^2+s)^{\frac{n-1}2}}\nonumber\\
=& - \frac{n(n-2)}{4b^n} \left(x_1^2 \int_0^\infty \frac{d\sigma}{(t+\sigma)^{3/2}(1+\sigma)^{\frac{n-1}2}} 
+ r^2 \int_0^\infty\frac{d\sigma}{\sqrt{t+\sigma}(1+\sigma)^{\frac{n+1}2}}\right) + C(a^2,b^2),
\end{align}
where $r^2=\sum_{i=2}^n x_i^2$, $C(a^2,b^2)$ is a constant that depends smoothly on $a^2$ and $b^2$, and in the last step we set $\sigma:=s/b^2$ and denoted with $t$ the aspect ratio $t:=a^2/b^2$, $t>0$ (see, e.g., \cite{Difratta}). In particular, $\Phi_0$ in $\Omega(a,b)$ is a second-degree polynomial with no linear terms. 

We now obtain the anisotropic term $\Psi$ of $\Phi_\alpha$ on $\Omega(a,b)$ (see \eqref{defPsi}) by differentiating $\Phi_0$ with respect to the aspect ratio $t$, in the spirit of \cite[Section 4]{CMMRSV}. First of all note that, by the definition of $\Phi_0$ and 
by a change of variables,
\begin{align}\label{Riesz:ellipsoid_nv2-n}
\Phi_0(b\sqrt t u_1,bu') = \frac1{b^{n-2}}\avint_{\!\!B_1(0)} \frac{dv}{\big(t(u_1-v_1)^2+|u'-v'|^2\big)^{\frac{n-2}2}},
\end{align}
where $u':=(u_2,\dots,u_n)$, and $u=(u_1,u')\in B_1(0)$. By differentiating \eqref{Riesz:ellipsoid_nv2-n} with respect to the aspect ratio $t$ we obtain
\begin{align}\label{diff:Phi0-n}
\frac{\partial}{\partial t} \Big(\Phi_0(b\sqrt t u_1,bu')\Big) = -\frac{1}{b^{n-2} }\frac{n-2}2 \avint_{\!\!B_1(0)} \frac{ (u_1-v_1)^2}{\big(t(u_1-v_1)^2+|u'-v'|^2\big)^{n/2}}\,dv ,
\end{align}
and since 
$$
\Psi(b\sqrt t u_1,bu') = \frac1{b^{n-2}} \avint_{\!\!B_1(0)} \frac{t (u_1-v_1)^2}{\big(t(u_1-v_1)^2+|u'-v'|^2\big)^{n/2}}\, dv,
$$
it follows by \eqref{diff:Phi0-n} that 
\begin{equation}\label{diff:anis-n}
\Psi(b\sqrt t u_1,bu') =-\frac{2t}{n-2} \frac{\partial}{\partial t} \Big(\Phi_0(b\sqrt t u_1,bu')\Big).
\end{equation}
On the other hand, by the explicit expression \eqref{Riesz:ellipsoid-n} we have that for every $u=(u_1,u')\in B_1(0)$ 
\begin{align*}
\Phi_0(b\sqrt t u_1,bu') = &- u_1^2 \, \frac{n(n-2) t}{4 b^{n-2}} \int_0^\infty \frac{d\sigma}{(t+\sigma)^{3/2}(1+\sigma)^{\frac{n-1}2}}\\
& - |u'|^2 \, \frac{n(n-2)}{4 b^{n-2}} \int_0^\infty\frac{d\sigma}{\sqrt{t+\sigma}(1+\sigma)^{\frac{n+1}2}} + C(tb^2,b^2),
\end{align*}
and hence 
\begin{align}\label{eq:RHS-n}
&\frac{\partial}{\partial t} \Big(\Phi_0(b\sqrt t u_1,bu') \Big)\nonumber\\
&=  - u_1^2 \, \frac{n(n-2)}{4 b^{n-2}} \left(\int_0^\infty \frac{d\sigma}{(t+\sigma)^{3/2}(1+\sigma)^{\frac{n-1}2}}
- \frac{3 t}2 \int_0^\infty \frac{d\sigma}{(t+\sigma)^{5/2}(1+\sigma)^{\frac{n-1}2}}\right)\nonumber\\
& + |u'|^2\,  \frac{n(n-2)}{8 b^{n-2}} \int_0^\infty\frac{d\sigma}{(t+\sigma)^{3/2}(1+\sigma)^{\frac{n+1}2}} +\tilde C(tb^2,b^2),
\end{align}
where $\tilde C(tb^2,b^2)$ is another constant. So, by \eqref{diff:anis-n} and \eqref{eq:RHS-n} we obtain the expression of $\Psi$ for $x\in\Omega(a,b)$, namely

\begin{align*}
\Psi(x) = &\, x_1^2 \, \frac{n}{2b^n} \left( \int_0^\infty \frac{d\sigma}{(t+\sigma)^{3/2}(1+\sigma)^{\frac{n-1}2}}
- \frac{3t}2 \int_0^\infty \frac{d\sigma}{(t+\sigma)^{5/2}(1+\sigma)^{\frac{n-1}2}}\right) \nonumber\\
& -  r^2 \, \frac{nt}{4 b^n} \int_0^\infty\frac{d\sigma}{(t+\sigma)^{3/2}(1+\sigma)^{\frac{n+1}2}}  -\frac{2t}{n-2} \tilde C(tb^2,b^2).
\end{align*}
In conclusion, for $x\in \Omega(a,b)$,
\begin{align}\label{potential:inside-n}
\Phi_\alpha(x)= &\,  x_1^2 \frac{n}{4 b^n} \bigg((2\alpha-(n-2))\int_0^\infty\frac{d\sigma}{(t+\sigma)^{3/2}(1+\sigma)^{\frac{n-1}2}} - 3\alpha t\int_0^\infty\frac{d\sigma}{(t+\sigma)^{5/2}(1+\sigma)^{\frac{n-1}2}}\bigg)\nonumber\\
 +& r^2 \frac{n}{4 b^n} \bigg(-(n-2)\int_0^\infty\frac{d\sigma}{\sqrt{t+\sigma} (1+\sigma)^{\frac{n+1}2}} - \alpha t\int_0^\infty\frac{d\sigma}{(t+\sigma)^{3/2}(1+\sigma)^{\frac{n+1}2}}\bigg) + C,
\end{align}
where $C$ denotes a constant and $t>0$.

\subsubsection{The condition \eqref{EL-1-real} on spheroids}\label{EL-1-real-spheroids-n}
We now use the expression \eqref{potential:inside-n} of the potential on spheroids to verify that there is a spheroid for which the first Euler-Lagrange condition  \eqref{EL-1-real}  is satisfied.

We start by establishing some relations among the integrals appearing in the expressions of the coefficients of $x_1^2$ and $r^2$. Note that, by defining
\begin{equation}\label{def:H}
H(t):=\int_0^\infty\frac{d\sigma}{(t+\sigma)^{3/2}(1+\sigma)^{\frac{n-1}2}}
\end{equation}
for $t>0$, we have that
\begin{align}
\int_0^\infty\frac{d\sigma}{(t+\sigma)^{5/2}(1+\sigma)^{\frac{n-1}2}} &=-\frac23 H'(t),\label{def:Hp}\\
\int_0^\infty\frac{d\sigma}{\sqrt{t+\sigma}(1+\sigma)^{\frac{n+1}2}} &=\frac{2}{n-1}\frac1{\sqrt{t}} - \frac1{n-1} H(t),\label{useful-b}\\
\int_0^\infty\frac{d\sigma}{(t+\sigma)^{3/2}(1+\sigma)^{\frac{n+1}2}} &=\frac{2}{n-1}\frac{1}{t^{3/2}} + \frac{2}{n-1}H'(t)\label{useful-c}.
\end{align}
Note also that 
\begin{equation}\label{HprimeH}
-n H(t)+2(1-t)H'(t) + \frac{2}{t^{3/2}}=0.
\end{equation}
To prove it, we rewrite \eqref{useful-c} as
\begin{align*}\label{H-rewrite-useful}
\frac{2}{n-1}H'(t) &= \frac1t \int_0^\infty\frac{t+\sigma -\sigma}{(t+\sigma)^{3/2}(1+\sigma)^{\frac{n+1}2}}\,d\sigma 
- \frac{2}{n-1}\frac{1}{t^{3/2}} 
\\
&=  \frac1t \left( \frac{2}{n-1}\frac1{\sqrt{t}} - \frac1{n-1} H(t) \right)
- \frac1t \int_0^\infty\frac{1+ \sigma -1}{(t+\sigma)^{3/2}(1+\sigma)^{\frac{n+1}2}}\,d\sigma 
- \frac{2}{n-1}\frac{1}{t^{3/2}} 
\\
&=  - \frac1{n-1} \frac1t H(t) 
- \frac1t H(t) +\frac1t \left( \frac{2}{n-1}\frac{1}{t^{3/2}} + \frac{2}{n-1}H'(t) \right),
\end{align*}
hence obtaining the relation \eqref{HprimeH}.

Integrating \eqref{HprimeH} we deduce that 
$$
H(t)=\frac1{|1-t|^{\frac{n}2}}\int_t^1 \frac{|1-s|^{\frac{n}2}}{s^{3/2}(1-s)}\, ds
\qquad \text{ for } t\neq1.
$$
By integration by parts we obtain
\begin{equation}\label{newH}
H(t)= \frac{2}{\sqrt t(1-t)}- \frac{n-2}{|1-t|^{n/2}}\int_t^1  \frac{|1-s|^{\frac{n}2-2}}{\sqrt s}\, ds \qquad \text{ for } t\neq1,
\end{equation}
and by differentiation
\begin{equation}\label{newH'}
H'(t)= \frac{(n+1)t-1}{t^{3/2}(1-t)^2}- \frac{n(n-2)}{2}\frac{1-t}{|1-t|^{\frac{n}2+2}}\int_t^1  \frac{|1-s|^{\frac{n}2-2}}{\sqrt s}\, ds \qquad \text{ for } t\neq1.
\end{equation}

Condition \eqref{EL-1-real} on spheroids is equivalent to the following two equations, obtained by equating the coefficients of $x_1^2$:
 \begin{equation}\label{r-EL1-1-n}
 -\frac12 = \frac{n}{4b^n}\left((2\alpha-(n-2))H(t)+2\alpha t H'(t)\right),
 \end{equation}
 and of $r^2$:
 \begin{equation}\label{r-EL1-2-n}
  -\frac12 = \frac{n}{4(n-1)b^n}\left(-2(n-2+\alpha)\frac1{\sqrt t} +(n-2) H(t)-2\alpha t H'(t)\right).
 \end{equation}
 
For what follows it is more convenient to reduce to two alternative equivalent conditions: the first one is obtained by adding \eqref{r-EL1-1-n} to $(n-1)$-times \eqref{r-EL1-2-n}:
\begin{equation}\label{r-EL1-4-n}
b^n = \frac{n-2+\alpha}{\sqrt t} - \alpha H(t),
\end{equation}
and the second condition is obtained by subtracting \eqref{r-EL1-1-n} from \eqref{r-EL1-2-n}:
\begin{align}\label{r-EL1-3-n}
\alpha \left((n-1)H(t)+nt H'(t)+\frac{1}{\sqrt t}\right) + (n-2)\left(-\frac{n}2 H(t)+\frac{1}{\sqrt t}\right)=0.
\end{align}
We claim that, for every $\alpha\in (-1,n-2]$ there exists a spheroid satisfying \eqref{r-EL1-4-n} and \eqref{r-EL1-3-n}, and hence satisfying the stationarity condition \eqref{EL-1-real}.

\subsubsection{The equation \eqref{r-EL1-3-n}}\label{sec313} We denote with $F(t,\alpha)$, for $\alpha\in \R$ and $t>0$ (we recall that $t=a^2/b^2$), the left-hand side of \eqref{r-EL1-3-n}. Then \eqref{r-EL1-3-n} is of the form $F(t,\alpha)=0$. We write
\begin{equation*}
F(t,\alpha) = \frac{1}{\sqrt t}\big(A(t)\alpha+B(t)\big), \quad (t,\alpha)\in (0,+\infty)\times \R,
\end{equation*}
where
\begin{align}\label{r-AandB-n}
 A(t):=  (n-1)\sqrt t H(t)+nt^{3/2} H'(t)+1, \qquad B(t):= -\frac{n(n-2)}2 \sqrt t H(t) + (n-2).
 \end{align}

We claim that for every $\alpha\in (-1,n-2]$ there exists $t=t(\alpha)>0$ such that $F(t(\alpha),\alpha)=0$. 
We start by analysing the behaviour of $A$ and $B$ for $t$ close to zero. By \eqref{newH} we have that 
\begin{equation}\label{newH2}
\sqrt t H(t) =  \frac{2}{1-t}- \frac{(n-2)\sqrt t}{(1-t)^{n/2}}\int_t^1  \frac{(1-s)^{\frac{n}2-2}}{\sqrt s}\, ds \qquad \text{ for } 0<t<1,
\end{equation}
and since 
\begin{equation}\label{integ-finite}
\int_0^1  \frac{(1-s)^{\frac{n}2-2}}{\sqrt s}\, ds<+\infty,
\end{equation}
we immediately deduce that $\lim_{t\to 0^+} \sqrt t H(t) = 2$. Moreover, by \eqref{newH'}, 
\begin{equation}\label{newH'2}
t^{3/2} H'(t) =  \frac{(n+1)t-1}{(1-t)^2}- \frac{n(n-2)t^{3/2}}{2(1-t)^{\frac{n}2+1}}\int_t^1  \frac{(1-s)^{\frac{n}2-2}}{\sqrt s}\, ds \qquad \text{ for } 0<t<1, 
\end{equation}
thus, by \eqref{integ-finite} we deduce that $\lim_{t\to 0^+} t^{3/2} H'(t) = -1$. This implies that
\begin{align*}
\lim_{t\to 0^+} A(t) = \lim_{t\to 0^+} \left((n-1)\sqrt t H(t)+ n t^{3/2} H'(t) + 1\right) = 2(n-1) - n +1 =n-1,
\end{align*}
and 
$$
\lim_{t\to 0^+} B(t)  = (n-2)\lim_{t\to 0^+} \left(-\frac{n}2 \sqrt t H(t) + 1\right) = -(n-1)(n-2).
$$
Hence, if $\alpha\neq n-2$, $\lim_{t \to 0^+} F(t,\alpha) = -\infty$. 
If $\alpha=n-2$, by \eqref{r-AandB-n} we have 
$$
\sqrt tF(t,n-2)= \frac{(n-2)^2}2 (\sqrt t H(t)-2) + n(n-2)(t^{3/2}H'(t)+1).
$$
Using \eqref{newH2} and \eqref{newH'2}, we obtain
$$
F(t,n-2)=\frac{(n-2)\sqrt t}{(1-t)^2}(2t+n^2-2) -
\frac{(n-2)^2}{2(1-t)^{\frac{n}2+1}}\left( (n-2)(1-t)+n^2t\right)
\int_t^1  \frac{(1-s)^{\frac{n}2-2}}{\sqrt s}\, ds
$$
for $0<t<1$.
By \eqref{integ-finite} we deduce that 
$$
\lim_{t \to 0^+} F(t, n-2)= -\frac{(n-2)^3}{2}\int_0^1  \frac{(1-s)^{\frac{n}2-2}}{\sqrt s}\, ds< 0.
$$

By a direct computation from \eqref{def:H} and \eqref{def:Hp} we have that $H(1) = \frac2n$ and $H'(1)=-\frac3{(n+2)}$, therefore  
$F(1,\alpha) = \frac{4(n-1)}{n(n+2)}\alpha$. Finally, one can check directly that 
\begin{equation}\label{limABinf}
\lim_{t\to +\infty}A(t)=1 \qquad \text{ and } \qquad \lim_{t\to +\infty}B(t)=n-2.
\end{equation}

Now fix $\bar \alpha\in (0,n-2]$: then, since $\lim_{t \to 0^+} F(t,\bar\alpha)< 0$ and $F(1,\bar\alpha)>0$, and $F(\cdot,\bar\alpha)$ is continuous
on $(0,\infty)$, it follows that there exists at least one $t(\bar\alpha)\in (0,1)$ such that $F(t(\bar\alpha),\bar\alpha)=0$. Similarly, fixing $\bar \alpha\in (-1,0)$, since $F(1,\bar\alpha)<0$ and $\lim_{t\to +\infty} F(t,\bar\alpha) = 0^+$, it follows that there exists at least one $t(\bar\alpha)>1$ such that $F(t(\bar\alpha),\bar\alpha)=0$.

In conclusion, for every $\alpha\in (-1,n-2]$ there exists at least one $t(\alpha)> 0$ such that $F(t(\alpha),\alpha) = 0$; in other words, for every $\alpha\in(-1,n-2]$ there exists a solution $t(\alpha)>0$ of \eqref{r-EL1-3-n}.

\subsubsection{The equation  \eqref{r-EL1-4-n}.}
Now we solve \eqref{r-EL1-4-n} for the $t(\alpha)$ found above by solving \eqref{r-EL1-3-n} for spheroids and we compute the corresponding $b$. Then the spheroid will be the one with semi-axes $b(\alpha)$ and $a(\alpha)$, with $t(\alpha)=a(\alpha)^2/b(\alpha)^2$. From \eqref{r-EL1-4-n} and by \eqref{useful-b}
\begin{equation*}
0< \sqrt t H(t) =2 - (n-1)\sqrt t \int_0^\infty\frac{d \sigma}{\sqrt{t+\sigma}(1+\sigma)^{\frac{n+1}2}} < 2 \qquad
\text{ for } t>0,
\end{equation*}
we have that, for $\alpha\in (0,n-2]$
\begin{align}\label{eqbpos}
b^n = \frac{1}{\sqrt t}\big(n-2+\alpha - \alpha\sqrt t H(t)\big) >  \frac{1}{\sqrt t}\big(n-2 - \alpha\big)\geq 0, 
\end{align}
and for $\alpha\in (-1,0)$
\begin{align}\label{eqbneg}
b^n = \frac{1}{\sqrt t}\big(n-2+\alpha - \alpha\sqrt t H(t)\big) > \frac{1}{\sqrt t}\big(n-2+\alpha) >0. 
\end{align}
\end{section}


\subsection{The condition \eqref{EL-2-real} outside spheroids}\label{sect:EL2-n} In this section we show that for $\alpha\in (-1,n-2]$ and for any spheroid $\Omega(a(\alpha),b(\alpha))$ for which the stationarity condition \eqref{EL-1-real} is satisfied, also the 
unilateral condition \eqref{EL-2-real} is satisfied. This implies that any spheroid $\Omega(a(\alpha),b(\alpha))$ for which the stationarity condition \eqref{EL-1-real} is satisfied is in fact a minimiser for the functional $I_\alpha$, and by Proposition \ref{exist+uniq} it is the unique minimiser (which in particular implies that there is only one spheroid satisfying \eqref{EL-1-real}).

We do it in several steps. We start by evaluating the Coulomb potential, which corresponds to $\alpha=0$, namely 
$$
\Phi_0(x)= \avint_{\!\!\Omega(a,b)}\frac{1}{|x-y|^{n-2}} \,dy,
$$
for $x\in \R^n\setminus\Omega(a,b)$ and $a,b>0$. For $x\in \R^n\setminus\Omega(a,b)$ and $a,b>0$,
\begin{align}\label{Riesz:noellipsoid-n}
\Phi_0(x) =  \frac{n(n-2)}4\int_{\lambda(x)}^\infty \left(1-\frac{x_1^2}{a^2+s} - \frac{r^2}{b^2+s}\right) \frac{ds}{\sqrt{a^2+s}(b^2+s)^{\frac{n-1}2}},
\end{align}
where $r^2=\sum_{i=2}^n x_i^2$ and $\lambda(x)$ is the largest root of the equation
$$
\frac{x_1^2}{a^2+\lambda} +\frac{r^2}{b^2+\lambda}=1.
$$
(see, e.g., \cite{Difratta}).
By straightforward computations one can see that
$$
\lambda(x)= 
\begin{cases}
\bigskip
\displaystyle \frac14\left(\sqrt{x_1^2+(r+c)^2} + \sqrt{x_1^2+(r-c)^2}\right)^2 - b^2 \quad &\textrm{if } a<b,\\
\displaystyle \frac14\left(\sqrt{(x_1+c)^2+r^2} + \sqrt{(x_1-c)^2+r^2}\right)^2 - a^2 \quad &\textrm{if } a>b,
\end{cases}
$$
where
$$
c^2:=\begin{cases}
b^2-a^2  \quad &\textrm{if } a<b,\\
a^2-b^2 \quad &\textrm{if } a>b.
\end{cases}
$$

\subsubsection{The anisotropic potential outside a spheroid}\label{sect:outside-n} In this section we prove that the anisotropic term $\Psi$ of $\Phi_\alpha$, both for oblate and prolate spheroids, is related to the Coulomb potential $\Phi_0$ by the relation
\begin{equation}\label{anis-explicit:E-n}
\Psi(x) = \frac{a^2}{a^2-b^2} \, \Phi_0(x)+ \frac{1}{(n-2)(a^2-b^2)} \nabla\Phi_0(x)\cdot \left( b^2 x_1, a^2 x'\right), \quad \textrm{for } x=(x_1,x')\in \R^n\setminus \Omega(a,b).
\end{equation}
We obtain \eqref{anis-explicit:E-n} by an ingenious differentiation of $\Phi_0$. While in Section \ref{sect:EL1} $\Psi$ on $\Omega(a,b)$ was obtained by differentiating $\Phi_0$ with respect to the aspect ratio $a^2/b^2$ of the spheroid, the geometric quantity that is relevant in this case is the parameter spanning a family of spheroids confocal with $\Omega(a,b)$, and surrounding it from the outside.

We prove \eqref{anis-explicit:E-n} in the case of oblate spheroids, but the case of prolate spheroids is completely analogous. For oblate spheroids we set $a^2=t$ and $b^2=t+c^2$, where $c$ is fixed, and set 
$$
\Phi^t_0(x):=\avint_{\!\!\Omega_t}\frac{1}{|x-y|^{n-2}} \,dy, \quad \Psi^t(x):=\avint_{\!\!\Omega_t}\frac{(x_1-y_1)^2}{|x-y|^n} \,dy, \quad \Omega_t:=\Omega(\sqrt t, \sqrt{t+c^2}). 
$$
From \eqref{Riesz:noellipsoid-n} one can easily rewrite the Coulomb potential on a spheroid as 
\begin{align*}
\Phi_0(x) =  \frac{n(n-2)}4\Big(\int_{\ell(x)}^\infty \frac{1}{\sqrt{\sigma -c^2}\sigma^{\frac{n-1}2}}\,d\sigma -\int_{\ell(x)}^\infty \frac{x_1^2}{(\sigma -c^2)^{3/2}\sigma^{\frac{n-1}2}}\,d\sigma 
- \int_{\ell(x)}^\infty \frac{r^2}{\sqrt{\sigma -c^2}\sigma^{\frac{n+1}2}}\,d\sigma\Big),
\end{align*}
where $\ell(x) = \lambda(x) + b^2$ and $\sigma:=s+b^2$, and $\ell(x)$ depends only on $c$. 
Hence $\Phi_0^t$ depends on its semi-axes $a$ and $b$ only via $c$, and not on $a$ and $b$ separately, namely
\begin{align*}
0=\frac{\partial}{\partial t} \Phi^t_0(x) 
= \frac{\partial}{\partial t}\avint_{\!\!B_1(0)}\!\! \bigg(t\bigg(\frac{x_1}{\sqrt t}-v_1\bigg)^2\!\!\!+(t+c^2)\sum_{i\neq 1}\bigg(\frac{x_i}{\sqrt{t+c^2}}-v_i\bigg)^2\bigg)^{-\frac{n-2}2}\!\!dv.
\end{align*}
By expanding the derivative above and rewriting 
\begin{multline*}
\bigg(\frac{x_1}{\sqrt t}-v_1\bigg)^2 + \sum_{i\neq 1}\bigg(\frac{x_i}{\sqrt{t+c^2}}-v_i\bigg)^2 \\
= \frac{1}{t+c^2}\bigg(t\bigg(\frac{x_1}{\sqrt t}-v_1\bigg)^2\!\!\!+(t+c^2)\sum_{i\neq 1}\bigg(\frac{x_i}{\sqrt{t+c^2}}-v_i\bigg)^2\bigg)
+ \frac{c^2}{t(t+c^2)} t \bigg(\frac{x_1}{\sqrt t}-v_1\bigg)^2,
\end{multline*}
we obtain 
\begin{equation}\label{derivative=0-n}
0=\frac{1}{t+c^2} \, \Phi_0^t(x) + \frac{c^2}{t(t+c^2)}\,\Psi^t(x) + \frac{1}{n-2}\nabla_x \Phi_0^t(x)\cdot \left(\frac{x_1}t, \frac{x'}{t+c^2}\right).
\end{equation}
The expression \eqref{derivative=0-n} gives the (unknown) expression of the anisotropic term $\Psi^t$ in terms of the (known) Coulomb potential $\Phi_0^t$ and its spatial gradient. Substituting $t$ and $c$ in terms of $a$ and $b$ in \eqref{derivative=0-n} and rearranging the terms we then have \eqref{anis-explicit:E-n}.

With the expression of the Coulomb potential $\Phi_0$ (see \eqref{Riesz:noellipsoid-n}) and a closed formula for the anisotropic potential $\Psi$ (in \eqref{anis-explicit:E-n}) in the outer region $\R^n\setminus \Omega(a,b)$ at hand, we now prove \eqref{EL-2-real}. More precisely, we prove that for every $\alpha\in (-1,n-2]$,
\begin{equation}\label{EL-2-s}
\Phi_\alpha(x) + \frac{|x|^2}2 \geq C_\alpha \quad \textrm{for every } x\in \R^n\setminus \Omega(a(\alpha),b(\alpha)),
\end{equation}
where $\Omega(a(\alpha),b(\alpha))$ is a stationary point, satisfying \eqref{EL-1-real}. To prove \eqref{EL-2-s} we  first rewrite the potentials $\Phi_0$ and $\Psi$ (and hence $\Phi_\alpha$) outside a spheroid in a more convenient way, in terms of a set of coordinates -- oblate or prolate spheroidal -- alternative to the Euclidean ones, and more suitable for the geometry of the problem. 

We deal  with the oblate and prolate case separately.

\subsubsection{The condition \eqref{EL-2-s} outside an oblate spheroid.}\label{checkEL2-n}

To prove \eqref{EL-2-s} for an oblate spheroid, we rewrite the potentials $\Phi_0$ in \eqref{Riesz:noellipsoid-n} and $\Psi$ in \eqref{anis-explicit:E-n} outside a spheroid in terms of the \textit{oblate spheroidal coordinates}. By the symmetry of $\Phi_\alpha$ and of the confinement, it is sufficient to reduce to computations in the $x_1 x_2$-plane. In terms of oblate spheroidal coordinates we have 
$$
\begin{cases}
x_1= cz\rho\\
x_2= c\sqrt{(1+z^2)(1-\rho^2)}
\end{cases}
\quad z\geq 0, \, \rho\in [-1,1],
$$
where we recall that $c^2= b^2-a^2$. Note that the outer region $\R^n\setminus \Omega(a,b)$ corresponds to $z\geq \frac{a}{c}$. 

For $z\geq \frac ac$ and $\rho\in [-1,1]$, the expression of the Coulomb potential \eqref{Riesz:noellipsoid-n} in oblate spheroidal coordinates reads as
\begin{align*}
\Phi_0(z,\rho) = \frac{n(n-2)}4 \rho^2\int_{c^2z^2}^\infty\frac{c^2(\sigma-c^2z^2)}{\sigma^{3/2}(\sigma+c^2)^{\frac{n+1}2}}\, d\sigma
+ \frac{n(n-2)}4\int_{c^2z^2}^\infty\frac{\sigma-c^2z^2}{\sqrt\sigma(\sigma+c^2)^{\frac{n+1}2}}\, d\sigma,
\end{align*}
where we used that $r^2=x_2^2$, $\lambda(x)=c^2z^2-a^2$ and the change of variables $\sigma=a^2+s$.
We recall that the gradient of the oblate spheroidal coordinates with respect to Cartesian coordinates is given by the following formulas:
\begin{eqnarray*}
\nabla \rho (x) & = & \dfrac1{c(z^2+\rho^2)}\big(z(1-\rho^2), -\rho\sqrt{(1+z^2)(1-\rho^2)}\big),
\\
\nabla z (x) &  = & \dfrac1{c(z^2+\rho^2)}\big(\rho(1+z^2), z\sqrt{(1+z^2)(1-\rho^2)}\big).
\end{eqnarray*}
Since 
\begin{align*}
\partial_z \Phi_0(z,\rho) &= -\frac{n(n-2)}2 c^2 z\int_{c^2z^2}^\infty\frac{c^2\rho^2+\sigma}{\sigma^{3/2}(\sigma+c^2)^{\frac{n+1}2}}\, d\sigma,\\
\partial_\rho \Phi_0(z,\rho) &= -\frac{n(n-2)}2 c^2 \rho\int_{c^2z^2}^\infty\frac{c^2 z^2-\sigma}{\sigma^{3/2}(\sigma+c^2)^{\frac{n+1}2}}\, d\sigma,
\end{align*}
we deduce that 
\begin{align*}
\nabla \Phi_0(x) = -\frac{c \, n (n-2)}2\left( \int_{c^2z^2}^\infty\frac{z\rho\, d\sigma}{\sigma^{3/2}(\sigma+c^2)^{\frac{n-1}2}}, 
\int_{c^2z^2}^\infty\frac{\sqrt{(1+z^2)(1-\rho^2)} \, d\sigma}{\sqrt\sigma(\sigma+c^2)^{\frac{n+1}2}}\right), 
\end{align*}
and 
\begin{multline*}
- \frac{1}{c^2} \nabla\Phi_0(x)\cdot(b^2x_1,a^2x_2) \\
= \frac{n(n-2)}{4} \left(\int_{c^2z^2}^\infty\frac{2b^2z^2\rho^2}{\sigma^{3/2}(\sigma+c^2)^{\frac{n-1}2}}\,  d\sigma+
\int_{c^2z^2}^\infty\frac{2a^2(1+z^2)(1-\rho^2)}{\sqrt\sigma(\sigma+c^2)^{\frac{n+1}2}}\,  d\sigma\right).
\end{multline*}
Hence the anisotropic potential is 
\begin{align*}
\Psi(x) =  & \frac{n}{4} \rho^2 \int_{c^2z^2}^\infty  
\frac{-na^2(\sigma-c^2z^2)+2c^2z^2(\sigma+c^2)}{\sigma^{3/2}(\sigma+c^2)^{\frac{n+1}2}}\,  d\sigma
\\
& +
\frac{n}{4} \int_{c^2z^2}^\infty 
\frac{-na^2(\sigma-c^2z^2)+2a^2(\sigma+c^2)}{c^2\sqrt{\sigma}(\sigma+c^2)^{\frac{n+1}2}}\,  d\sigma .
\end{align*}
Finally, the confinement term $|x|^2/2$, in terms of the spheroidal coordinates, is 
$$
\frac{|x|^2}{2}
= \frac{c^2}2(1-\rho^2+z^2).
$$
Note that $\Phi_0$, $\Psi$ and the confinement are all quadratic functions in the variable $\rho$. More precisely, for $z\geq \frac ac$ and $\rho\in [-1,1]$,
\begin{align}\label{quadratic:outside+}
\Phi_\alpha(x)+\frac{|x|^2}{2}
= A_\alpha(z)+B_\alpha(z)\rho^2,
\end{align}
where 
\begin{align*}
A_\alpha(z):=& \frac{n}{4} \int_{c^2z^2}^\infty \frac{\big( (n-2)c^2 - n\alpha  a^2\big)(\sigma-c^2z^2)
+ 2\alpha a^2 (\sigma+c^2)}{c^2\sqrt{\sigma}(\sigma+c^2)^{\frac{n+1}2}}\,  d\sigma + \frac{c^2}2(1+z^2),\\
B_\alpha(z):=& \frac{n}{4} \int_{c^2z^2}^\infty \frac{\big( (n-2)c^2 - n\alpha  a^2\big)(\sigma-c^2z^2)
+ 2\alpha c^2z^2 (\sigma+c^2)}{\sigma^{3/2}(\sigma+c^2)^{\frac{n+1}2}}\,  d\sigma - \frac{c^2}2.
\end{align*}
By \eqref{quadratic:outside+}, proving \eqref{EL-2-s} is now equivalent to show that
\begin{equation}\label{MG-equiv}
A_\alpha(z)+B_\alpha(z)\geq C_\alpha \quad \textrm{and } \quad A_\alpha(z)\geq C_\alpha \quad \textrm{for } z\geq \frac ac,
\end{equation}
namely to check the inequality for $\rho=0$ and $\rho=1$. Moreover, note that the functions $A_\alpha$ and $B_\alpha$ are well-defined and smooth at every $z>0$. Since the potential $\Phi_\alpha + |\cdot|^2/2$ belongs to $C^1(\R^2)$ and satisfies \eqref{EL-1-real}, equation \eqref{quadratic:outside+} implies that
\begin{equation}\label{subc}
A_\alpha\Big(\frac{a}{c}\Big)=A_\alpha\Big(\frac{a}{c}\Big)+ B_\alpha\Big(\frac{a}{c}\Big)=C_\alpha
\qquad \text{ and } \qquad A_\alpha'\Big(\frac{a}{c}\Big)=A_\alpha'\Big(\frac{a}{c}\Big)+ B'_\alpha\Big(\frac{a}{c}\Big)=0.
\end{equation}

We show that \eqref{MG-equiv} is satisfied by proving that 
\begin{equation}\label{claimABz2}
\Big(\frac1z \big(A'_\alpha(z)+B'_\alpha(z)\big)\Big)'\geq 0 \quad \textrm{and } \quad \Big(\frac1z A'_\alpha(z)\Big)'\geq 0 \quad \textrm{for } z\geq \frac ac.
\end{equation}
Clearly \eqref{claimABz2}, together with the second condition in \eqref{subc}, implies 
$$
A'_\alpha(z)+ B'_\alpha(z)\geq 0 \quad \textrm{and } \quad A'_\alpha(z)\geq 0 \quad \textrm{for } z\geq \frac ac,
$$
which, in turn, gives \eqref{MG-equiv} owing to the first condition in \eqref{subc}.
While \eqref{claimABz2} may look more complicated, it actually gives rise to simpler computations. Indeed we have 
$$
 \Big(\frac1z A'_\alpha(z)\Big)' = \frac{n}{c^{n-2} z^2(1+z^2)^{\frac{n+1}2}}\Big((n-2)z^2+\alpha \frac{a^2}{c^2}\Big) \geq 0 \quad \textrm{for every } z\geq \frac ac,
$$
since $n\geq 3$ and $\alpha\geq 0$. Moreover,
\begin{align*}
\Big(\frac1z \big(A'_\alpha(z)+B'_\alpha(z)\big)\Big)' = \frac{n}{c^{n-2} z^2(z^2+1)^{\frac{n+1}2}} \Big(
(n-2)(1+\alpha)z^2 +n-2-\alpha -\alpha \frac{a^2}{c^2}(n-1)\Big).
\end{align*}
Hence to prove the claim \eqref{claimABz2} it is sufficient to show that 
$$
(n-2)(1+\alpha)z^2 +n-2-\alpha -\alpha \frac{a^2}{c^2}(n-1)\geq0
 \qquad \text{ for } z\geq \frac ac.
$$
This condition is satisfied since
$$
(n-2)(1+\alpha) \frac{a^2}{c^2} +n-2-\alpha -\alpha \frac{a^2}{c^2}(n-1)
= (n-2-\alpha)\left( 1+\frac{a^2}{c^2}\right)\geq0.
$$


\subsubsection{The condition \eqref{EL-2-s} outside a prolate spheroid.}\label{checkEL2-n-pro}

To prove \eqref{EL-2-s} for a prolate spheroid, we rewrite the potentials $\Phi_0$ in \eqref{Riesz:noellipsoid-n} and $\Psi$ in \eqref{anis-explicit:E-n} outside a spheroid in terms of the \textit{prolate spheroidal coordinates}. By the symmetry of $\Phi_\alpha$ and of the confinement, it is sufficient to reduce to computations in the $x_1 x_2$-plane. In terms of prolate spheroidal coordinates we have 
$$
\begin{cases}
x_1= cz\rho\\
x_2= c\sqrt{(z^2-1)(1-\rho^2)}
\end{cases}
\quad z\geq 1, \, \rho\in [-1,1],
$$
where we recall that now $c^2= a^2-b^2$. The outer region $\R^n\setminus \Omega(a,b)$ corresponds also in this case to $z\geq \frac{a}{c}$. 

For $z\geq \frac ac$ and $\rho\in [-1,1]$, the expression of the Coulomb potential \eqref{Riesz:noellipsoid-n} in prolate spheroidal coordinates reads as
\begin{align*}
\Phi_0(z,\rho) = \frac{n(n-2)}4 \rho^2\int_{c^2z^2}^\infty\frac{c^2(c^2z^2-\sigma)}{\sigma^{3/2}(\sigma-c^2)^{\frac{n+1}2}}\, d\sigma
+ \frac{n(n-2)}4\int_{c^2z^2}^\infty\frac{\sigma-c^2z^2}{\sqrt\sigma(\sigma-c^2)^{\frac{n+1}2}}\, d\sigma,
\end{align*}
where we used that $r^2=x_2^2$, $\lambda(x)=c^2z^2-a^2$ and a change of variables.
We recall that the gradient of the prolate spheroidal coordinates with respect to Cartesian coordinates is given by the following formulas:
\begin{eqnarray*}
\nabla \rho (x) & = & \dfrac1{c(z^2-\rho^2)}\big(z(1-\rho^2), -\rho\sqrt{(z^2-1)(1-\rho^2)}\big),
\\
\nabla z(x) &  = & \dfrac1{c(z^2-\rho^2)}\big(\rho(z^2-1), z\sqrt{(z^2-1)(1-\rho^2)}\big).
\end{eqnarray*}
Since 
\begin{align*}
\partial_z \Phi_0(z,\rho) &= \frac{n(n-2)}2 c^2 z\int_{c^2z^2}^\infty\frac{c^2\rho^2-\sigma}{\sigma^{3/2}(\sigma-c^2)^{\frac{n+1}2}}\, d\sigma,\\
\partial_\rho \Phi_0(z,\rho) &= \frac{n(n-2)}2 c^2 \rho\int_{c^2z^2}^\infty\frac{c^2 z^2-\sigma}{\sigma^{3/2}(\sigma-c^2)^{\frac{n+1}2}}\, d\sigma,
\end{align*}
we deduce that 
\begin{align*}
\nabla \Phi_0(x) = -\frac{c \, n (n-2)}2\left( \int_{c^2z^2}^\infty\frac{z\rho\, d\sigma}{\sigma^{3/2}(\sigma-c^2)^{\frac{n-1}2}}, 
\int_{c^2z^2}^\infty\frac{\sqrt{(z^2-1)(1-\rho^2)} \, d\sigma}{\sqrt\sigma(\sigma-c^2)^{\frac{n+1}2}}\right), 
\end{align*}
and 
\begin{multline*}
\frac{1}{c^2} \nabla\Phi_0(x)\cdot(b^2x_1,a^2x_2) \\
= -\frac{n(n-2)}{4} \left(\int_{c^2z^2}^\infty\frac{2b^2z^2\rho^2}{\sigma^{3/2}(\sigma-c^2)^{\frac{n-1}2}}\,  d\sigma+
\int_{c^2z^2}^\infty\frac{2a^2(z^2-1)(1-\rho^2)}{\sqrt\sigma(\sigma-c^2)^{\frac{n+1}2}}\,  d\sigma\right).
\end{multline*}
Hence the anisotropic potential is 
\begin{align*}
\Psi(x) =  & \frac{n}{4} \rho^2 \int_{c^2z^2}^\infty  
\frac{-na^2(\sigma-c^2z^2)+2c^2z^2(\sigma-c^2)}{\sigma^{3/2}(\sigma-c^2)^{\frac{n+1}2}}\,  d\sigma
\\
& +
\frac{n}{4} \int_{c^2z^2}^\infty 
\frac{na^2(\sigma-c^2z^2)-2a^2(\sigma-c^2)}{c^2\sqrt{\sigma}(\sigma-c^2)^{\frac{n+1}2}}\,  d\sigma .
\end{align*}

Finally, the confinement term $|x|^2/2$, in terms of the spheroidal coordinates, is 
$$
\frac{|x|^2}{2}
= \frac{c^2}2(-1+\rho^2+z^2).
$$
Note that, as before, $\Phi_0$, $\Psi$ and the confinement are all quadratic functions in the variable $\rho$. More precisely, for $z\geq \frac ac$ and $\rho\in [-1,1]$,
\begin{align}\label{quadratic:outside1}
\Phi_\alpha(x)+\frac{|x|^2}{2}
= A_\alpha(z)+B_\alpha(z)\rho^2,
\end{align}
where 
\begin{align*}
A_\alpha(z):=& \frac{n}{4} \int_{c^2z^2}^\infty \frac{\big( (n-2)c^2 + n\alpha  a^2\big)(\sigma-c^2z^2)
- 2\alpha a^2 (\sigma-c^2)}{c^2\sqrt{\sigma}(\sigma-c^2)^{\frac{n+1}2}}\,  d\sigma + \frac{c^2}2(z^2-1),\\
B_\alpha(z):=& \frac{n}{4} \int_{c^2z^2}^\infty \frac{\big(-(n-2)c^2 - n\alpha  a^2\big)(\sigma-c^2z^2)
+ 2\alpha c^2z^2 (\sigma-c^2)}{\sigma^{3/2}(\sigma-c^2)^{\frac{n+1}2}}\,  d\sigma + \frac{c^2}2.
\end{align*}
Note that, by \eqref{quadratic:outside1}, proving \eqref{EL-2-s} is equivalent to show that
$$
A_\alpha(z)+B_\alpha(z)\geq C_\alpha \quad \textrm{and } \quad A_\alpha(z)\geq C_\alpha \quad \textrm{for } z\geq \frac ac,
$$
namely to check the inequality for $\rho=0$ and $\rho=1$. 
Arguing as in the case of an oblate spheroid, it is in fact sufficient to prove that
\begin{equation}\label{claimABz-22}
\Big(\frac1z \big(A'_\alpha(z)+B'_\alpha(z)\big)\Big)'\geq 0 \quad \textrm{and } \quad \Big(\frac1z A'_\alpha(z)\Big)'\geq 0 \quad \textrm{for } z\geq \frac ac.
\end{equation}
We have 
$$
 \Big(\frac1z A'_\alpha(z)\Big)' = \frac{n}{c^{n-2} z^2(z^2-1)^{\frac{n+1}2}}\Big((n-2)z^2+\alpha \frac{a^2}{c^2}\Big) \geq 0 \quad \textrm{for every } z\geq \frac ac,
$$
since $n\geq 3$ and $\alpha>-1$. Moreover,
\begin{align*}
\Big(\frac1z \big(A'_\alpha(z)+B'_\alpha(z)\big)\Big)' = \frac{n}{c^{n-2} z^2(z^2-1)^{\frac{n+1}2}} \Big(
(n-2)(1+\alpha)z^2 -(n-2)+\alpha -\alpha \frac{a^2}{c^2}(n-1)\Big).
\end{align*}
Hence to prove the claim \eqref{claimABz-22} it is sufficient to show that 
$$
(n-2)(1+\alpha)z^2 -(n-2)+\alpha -\alpha \frac{a^2}{c^2}(n-1)\geq0
 \qquad \text{ for } z\geq \frac ac.
$$
This condition is satisfied because $\alpha>-1$ and
$$
(n-2)(1+\alpha) \frac{a^2}{c^2} -n+2+\alpha -\alpha \frac{a^2}{c^2}(n-1)
= (n-2-\alpha)\left(\frac{a^2}{c^2}-1\right)\geq0.
$$

This concludes the proof of Theorem~\ref{characterisation}.

\begin{remark}\label{rmk:ta-ci}
For every $\alpha\in (-1,n-2]$, let $t(\alpha)>0$ be the solution of the equation $F(t,\alpha)=0$ found in Section~\ref{sec313}.
Note that this solution is unique. Indeed, by \eqref{eqbpos}--\eqref{eqbneg}, for a given $\alpha\in (-1,n-2]$, any solution $t(\alpha)>0$ of $F(t,\alpha)=0$ identifies a spheroid 
satisfying the stationarity condition \eqref{EL-1-real}. Moreover, in Section~\ref{sect:EL2-n} we show that any stationary spheroid satisfies also condition \eqref{EL-2-real}, so it 
is a minimiser of $I_\alpha$, and hence is unique by strict convexity. 
Moreover, the function
$t: \alpha\in (-1,n-2] \mapsto t(\alpha)\in (0,+\infty)$ is continuous 
and strictly decreasing. To prove it, let $\alpha_0\in (-1,n-2]$ and let $\alpha\to \alpha_0$. Let $(\alpha_k)$ denote a subsequence converging monotonically to $\alpha_0$, as $k\to +\infty$; note that the subsequence $(I_{\alpha_k})$ is also monotone. Since $I_{\alpha_k}$ is lower semicontinuous for every $k$,  we have that $I_{\alpha_k}$ $\Gamma$-converges to $I_{\alpha_0}$, as $k\to+\infty$, with respect to narrow convergence (see, e.g., \cite[Propositions~5.4 and~5.7]{DM}). By the Urysohn property of $\Gamma$-convergence (see for instance \cite[Proposition 8.3]{DM}) the whole sequence $(I_\alpha)$ $\Gamma$-converges to $I_{\alpha_0}$, as $\alpha\to\alpha_0$, since the space of probability measures endowed with the narrow convergence is metrisable. Moreover, the equilibrium measures $\mu_\alpha$ satisfy
$$
\int_{\R^n}|x|^2\, d\mu_\alpha(x)\leq I_\alpha(\mu_\alpha)\leq I_\alpha(\mu_{n-2})\leq  I_{n-2}(\mu_{n-2}) \qquad
\text{ for every } \alpha\in(-1,n-2],
$$
where we used the minimality of $\mu_\alpha$. In other words, the equilibrium measures are tight.
By the Fundamental Theorem of $\Gamma$-convergence (see, e.g., \cite[Corollary~7.17]{DM}) and the uniqueness of the minimiser of $I_{\alpha_0}$
we deduce that $\mu_\alpha$ converges narrowly to $\mu_{\alpha_0}$, as $\alpha\to\alpha_0$.
The characterisation of $\mu_\alpha$ given in Theorem~\ref{characterisation} allows us to conclude that $t(\alpha)\to t(\alpha_0)$, as $\alpha\to\alpha_0$, that is, the function $\alpha\mapsto t(\alpha)$ is continuous. It is also injective by the following argument: If $\alpha_1,\alpha_2\in(-1,n-2]$, 
$\alpha_1\neq\alpha_2$, are such that $t(\alpha_1)=t(\alpha_2)=t_0>0$, then $F(t_0,\alpha_1)=0=F(t_0,\alpha_2)$,
which implies $A(t_0)=0$, with $A$ defined in \eqref{r-AandB-n}. Since $F(t_0,\alpha_1)=0$, we deduce that $B(t_0)=0$, but this would imply that $F(t_0,\alpha)=0$
for every $\alpha\in (-1,n-2]$, that is, the function $\alpha\mapsto t(\alpha)$ is constant. This is not possible, since $t(\alpha)>1$ for $\alpha<0$ and 
$t(\alpha)<1$ for $\alpha>0$. This last property, together with continuity and injectivity, implies that $\alpha\mapsto t(\alpha)$ is strictly decreasing. 
\end{remark}

\begin{section}{The limiting case, as $\alpha\to-1$}\label{Section-1}

In this section we discuss the behaviour of the nonlocal energies $I_\alpha$, as  $\alpha\to-1$.

\begin{theorem}\label{thm:-1}
As $\alpha\to-1^+$, the functionals $I_\alpha$ $\Gamma$-converge, with respect to narrow convergence, to a functional
$J_\ast:\mathcal P(\R^n)\to[0,+\infty]$, whose unique minimiser is a measure $\mu_\ast$ 
that is the normalised characteristic function of a prolate spheroid $\Omega_\ast$.
Moreover, the following representation holds:
\begin{equation}\label{Jast-rep}
J_\ast (\mu)= \int_{\R^n} \widehat{W}_\ast(\xi) |\widehat\mu(\xi)|^2\, d\xi + \int_{\mathbb{R}^n}
|x|^2  \,d\mu(x)
\end{equation}
for every $\mu\in \mathcal P(\R^n)$ with compact support, where
$\widehat W_\ast\in L^1_{\text{loc}}(\R^n)$ is the function given by
$$
\widehat W_\ast(\xi):=\frac{\pi^{\frac n2-2}}{2\Gamma(\frac n2)}  \frac{(n-1) \xi_1^2 + (n-3) \sum_{i=2}^n \xi_i^2}{|\xi|^4}  
$$
for a.e.\ $\xi\in\R^n$. 
\end{theorem}

\begin{remark}\label{asymrem}
The functional $J_\ast$ does not coincide with the functional $I_{-1}$ defined in Remark~\ref{alfaneg}, since $J_\ast$ is lower semicontinuous, whereas $I_{-1}$ is not. Moreover $J_\ast$ does not coincide with the functionals $\overline{I_{-1}}$ or $\widetilde I_{-1}$ either, since the Dirac delta at $0$ is a minimiser for both $\overline{I_{-1}}$ and $\widetilde I_{-1}$.
\end{remark}

\begin{proof}[Proof of Theorem~\ref{thm:-1}]
The proof is subdivided into several steps.\smallskip

\noindent \textit{Step~1: Convergence of the equilibrium measures.} 
We prove that, as $\alpha\to-1^+$, the equilibrium measures $\mu_\alpha$ converge
to a measure $\mu_\ast$, that is the normalised characteristic function of a prolate spheroid $\Omega_\ast$.
First of all, we claim that 
\begin{equation}\label{tast}
\lim_{\alpha\to-1^+} t(\alpha)=t_\ast \in\R
\end{equation}
for some $t_\ast>1$. 
Since $\alpha\mapsto t(\alpha)$ is strictly decreasing by Remark~\ref{rmk:ta-ci} and $t(\alpha)>1$ for $\alpha<0$,
the limit as $\alpha\to-1^+$ exists and is strictly greater than $1$.
Assume by contradiction that $t_\ast=+\infty$. 
By \eqref{limABinf} we can pass to the limit in the equation
\begin{equation}\label{nn}
0=\sqrt{t(\alpha)} F(t(\alpha), \alpha)= A(t(\alpha))\alpha + B(t(\alpha)) 
\end{equation}
and deduce that $n-3=0$, which gives a contradiction for $n>3$. If $n=3$, equation \eqref{nn}, together with the assumption that $t_\ast=+\infty$, implies that
\begin{equation}\label{oldC}
-A(t(\alpha))+ B(t(\alpha))<0 
\end{equation}
for $\alpha+1>0$ small enough (note that $A(t(\alpha))>0$ for $\alpha+1>0$ small enough by \eqref{limABinf}).
By \eqref{newH} and \eqref{newH'}, for $t>1$ and $n=3$ we obtain
$$
-A(t)+ B(t)= -\frac{5t+4}{(t-1)^2}+\frac{\sqrt t}{2(t-1)^{5/2}}(2t+7)\int_1^t  \frac1{\sqrt{s(s-1)}}\, ds.
$$
Since the right-hand side is positive as $t\to+\infty$, this contradicts \eqref{oldC}.
Claim \eqref{tast} is thus proved for every $n\geq3$.
Passing to the limit in \eqref{eqbneg} we also obtain that
$$
\lim_{\alpha\to-1^+} b(\alpha) = b_\ast:=(t_\ast)^{-\frac{1}{2n}}\big(n-3+\sqrt{t_\ast}H(t_\ast)\big)^{\frac1n}>0. 
$$
This implies that the equilibrium measures $\mu_\alpha$ converge narrowly, as $\alpha\to-1^+$, to
the normalised characteristic function of the prolate spheroid $\Omega_\ast=\Omega(\sqrt{t_\ast}b_\ast,b_\ast)$.\smallskip

\noindent \textit{Step~2: $\Gamma$-convergence.} 
Let $(\alpha_k)$ be a sequence such that $\alpha_k\to-1^+$, as $k\to+\infty$.
Note that we can extract a decreasing subsequence $(\alpha_{k_j})\searrow -1^+$ along which also the sequence of functionals $(I_{\alpha_{k_j}})$ is decreasing, and hence $\Gamma$-convergent as $j\to +\infty$ to the functional
$$
J_\ast(\mu):=\overline J(\mu), \qquad J(\mu):=\inf_{\alpha\in(0,-1)} I_\alpha(\mu) 
$$
for every $\mu\in\mathcal P(\R^n)$, where $\overline J$ denotes the lower semicontinuous envelope of $J$ with respect to narrow convergence. By the Urysohn property of $\Gamma$-convergence (see for instance \cite[Proposition 8.3]{DM}) the whole sequence $(I_\alpha)$ $\Gamma$-converges to $J_\ast$, as $\alpha\to -1^+$.
By the Fundamental Theorem of $\Gamma$-convergence we deduce that $\mu_\ast$ is a minimiser of $J_\ast$ and
\begin{equation}\label{fundG}
\lim_{\alpha\to-1^+} I_{\alpha}(\mu_{\alpha})= J_\ast(\mu_\ast).
\end{equation} 

\noindent \textit{Step~3: Representation formula for $J_\ast$.} 
Using formula \eqref{Fourier:space}, which holds for every $\nu\in\mathcal P(\R^n)$ with compact support and every $\alpha\in(-1,n-2]$, we can now prove the representation formula~\eqref{Jast-rep}.

We first observe that by \eqref{FT3d}
\begin{equation}\label{aeW}
\widehat W_\alpha(\xi)\to\widehat W_\ast(\xi) \qquad \text{ for a.e.\ } \xi\in\R^n,
\end{equation}
as $\alpha\to-1^+$, and there exists a constant $C$, independent of $\alpha$, such that
\begin{equation}\label{dominW}
0\leq \widehat W_\alpha\leq C \widehat W_0
\end{equation}
for every $\alpha\in (-1,n-2]$.
Let now $\nu\in \mathcal P(\R^n)$ be a measure with compact support. Let $(\alpha_k)\subset (-1,n-2]$ be any sequence converging to $-1$, and 
let $(\nu_k)$ be a sequence in $\mathcal P(\R^n)$ converging narrowly to $\nu$ and such that
$$
\liminf_{k\to\infty} I_{\alpha_k}(\nu_k)<+\infty.
$$
Up to subsequences, we can assume that $\sup_k I_{\alpha_k}(\nu_k)<+\infty$. By \eqref{bound:compact}
there exists a compact set $K\subset\R^n$, containing the support of $\nu$, such that $\nu_k(K)>0$ and
$$
W_{\alpha_k}(x-y) + \frac12(|x|^2+|y|^2)\geq  I_{\alpha_k}(\nu_k) +1 \qquad
\text{ for } (x,y)\not\in K\times K
$$
for every $k$. If we define 
\begin{equation}\label{defmuK}
\mu_k:=\frac{\nu_k\mres K}{\nu_k(K)}
\end{equation}
for every $k$, then $\mu_k\in \mathcal P(\R^n)$ has compact support and
\begin{eqnarray*}
I_{\alpha_k}(\nu_k) & =  &
\big(\nu_k(K)\big)^2 I_{\alpha_k}(\mu_k) +  \iint_{(K\times K)^c}\Big( W_{\alpha_k}(x-y) + \frac12(|x|^2+|y|^2)\Big)\, d\nu_k(x)\,d\nu_k(y)
\\
& \geq  & \big(\nu_k(K)\big)^2 I_{\alpha_k}(\mu_k) + \big(1- \big(\nu_k(K)\big)^2\big)(I_{\alpha_k}(\nu_k)+1).
\end{eqnarray*}
This inequality implies that
\begin{equation}\label{nm100}
I_{\alpha_k}(\mu_k)\leq I_{\alpha_k}(\nu_k) -\frac{1- \big(\nu_k(K)\big)^2}{\big(\nu_k(K)\big)^2}\leq I_{\alpha_k}(\nu_k)
\end{equation}
for every $k$. Since $(\mu_k)$ converges narrowly to $\nu$, as $k\to\infty$,
we have that $(\widehat\mu_k)$ pointwise converges to $\widehat\nu$ and by the Fatou Lemma
$$
\liminf_{k\to\infty} \int_{\R^n}
\widehat{W}_{\alpha_k}(\xi) |\widehat\mu_k(\xi)|^2\, d\xi \geq \int_{\R^n}
\widehat{W}_\ast(\xi) |\widehat\nu(\xi)|^2\, d\xi.
$$
By \eqref{nm100}, \eqref{Fourier:space}, and the continuity of the confinement term with respect to narrow convergence (on measures with compact support),
we obtain
$$
\liminf_{k\to\infty} I_{\alpha_k}(\nu_k) \geq \liminf_{k\to\infty} I_{\alpha_k}(\mu_k) \geq 
\int_{\R^n} \widehat{W}_\ast(\xi) |\widehat\nu(\xi)|^2\, d\xi + \int_{\mathbb{R}^n}
|x|^2  \,d\nu(x).
$$
Since by definition of $\Gamma$-convergence 
\begin{equation}\label{defG}
J_\ast (\nu)=\min\Big\{ \liminf_{k\to\infty} I_{\alpha_k}(\nu_k): \ (\nu_k)\rightharpoonup\nu \text{ narrowly}, (\alpha_k)\to-1^+\Big\},
\end{equation}
we deduce that
\begin{equation}\label{Gliminf}
J_\ast (\nu)\geq \int_{\R^n} \widehat{W}_\ast(\xi) |\widehat\nu(\xi)|^2\, d\xi + \int_{\mathbb{R}^n}
|x|^2  \,d\nu(x)
\end{equation}
for every $\nu\in \mathcal P(\R^n)$ with compact support.

On the other hand, to prove the opposite inequality in \eqref{Gliminf}, let us first consider $\nu\in \mathcal P(\R^n)\cap C^\infty_c(\R^n)$. By \eqref{dominW} we have that 
$$
0\leq \widehat{W}_{\alpha_k}(\xi) |\widehat\nu(\xi)|^2\leq C\widehat{W}_0(\xi) |\widehat\nu(\xi)|^2,
$$
which gives a domination in $L^1(\R^n)$, since $\widehat\nu\in\mathcal S$ and $\widehat{W}_0\in L^1_{\textrm{loc}}(\R^n)$ behaves as $1/|\xi|^2$ at infinity.
Therefore, by \eqref{aeW} and the Dominated Convergence Theorem
$$
\lim_{k\to\infty} \int_{\R^n}
\widehat{W}_{\alpha_k}(\xi) |\widehat\nu(\xi)|^2\, d\xi = \int_{\R^n}
\widehat{W}_\ast(\xi) |\widehat\nu(\xi)|^2\, d\xi
$$
for every $\nu\in \mathcal P(\R^n)\cap C^\infty_c(\R^n)$.
By \eqref{defG} and \eqref{Fourier:space} this implies that
$$
J_\ast (\nu)\leq  \int_{\R^n}
\widehat{W}_\ast(\xi) |\widehat\nu(\xi)|^2\, d\xi + \int_{\mathbb{R}^n}
|x|^2  \,d\nu(x)
$$
for every $\nu\in \mathcal P(\R^n)\cap C^\infty_c(\R^n)$.
Let now $\nu\in \mathcal P(\R^n)$ be a measure with compact support and let $\nu_\e$ be defined as in \eqref{convol}. Using the inequality above and the lower semicontinuity of $J_\ast$
with respect to narrow convergence, we obtain that
\begin{equation}\label{Glimsup1}
J_\ast (\nu)\leq \liminf_{\e\to0} J_\ast (\nu_\e)
\leq \liminf_{\e\to0} \int_{\R^n}
\widehat{W}_\ast(\xi) |\widehat{\nu_\e}(\xi)|^2\, d\xi + \int_{\mathbb{R}^n}
|x|^2  \,d\nu(x).
\end{equation}
Arguing exactly as in \eqref{trickJV}, we have that
\begin{equation}\label{Glimsup2}
\lim_{\e\to0} \int_{\R^n}
\widehat{W}_\ast(\xi) |\widehat{\nu_\e}(\xi)|^2\, d\xi
= \int_{\R^n} \widehat{W}_\ast(\xi) |\widehat\nu(\xi)|^2\, d\xi,
\end{equation}
even if the right-hand side is infinite.
Combining \eqref{Gliminf}--\eqref{Glimsup2}, we finally obtain \eqref{Jast-rep}.\smallskip

\noindent \textit{Step~4: Uniqueness of the minimiser of $J_\ast$.} 
We now use the representation \eqref{Jast-rep} to show that the minimiser of $J_\ast$ is in fact unique.

First of all, we prove that any minimiser of $J_\ast$ must have compact support. Indeed, assume by contradiction that $\nu_\ast$ is a minimiser of $J_\ast$
not compactly supported. Let $(\alpha_k)\subset (-1,n-2]$ be a sequence converging to $-1$ and let $(\nu_k)\subset\mathcal P(\R^n)$ be a recovery sequence for $\nu_\ast$, that is, such that $(\nu_k)$ converges to $\nu_\ast$
narrowly and
$$
\lim_{k\to\infty} I_{\alpha_k}(\nu_k)=J_\ast(\nu_\ast).
$$
In particular, $\sup_k I_{\alpha_k}(\nu_k)<+\infty$. We argue in a similar way as in \eqref{nm100}. By \eqref{bound:compact}
there exists a compact set $K\subset\R^n$ such that $0<\nu_\ast(K)<1$, $\nu_k(K)>0$ and
$$
W_{\alpha_k}(x-y) + \frac12(|x|^2+|y|^2)\geq  I_{\alpha_k}(\nu_k) +1 \qquad
\text{ for } (x,y)\not\in K\times K
$$
for every $k$. If we define $\mu_k$ as in \eqref{defmuK}, then 
$$
I_{\alpha_k}(\mu_k)\leq I_{\alpha_k}(\nu_k) -\frac{1- \big(\nu_k(K)\big)^2}{\big(\nu_k(K)\big)^2}
$$
for every $k$. Since by narrow convergence $\nu_k(K)\to\nu_\ast(K)$, as $k\to\infty$, we obtain
\begin{equation}\label{abs1001}
\liminf_{k\to\infty} I_{\alpha_k}(\mu_k)\leq J_\ast(\nu_\ast) -\frac{1- \big(\nu(K)\big)^2}{\big(\nu(K)\big)^2}
< J_\ast(\nu_\ast).
\end{equation}
On the other hand, by the minimality of $\mu_{\alpha_k}$ and by \eqref{fundG}
\begin{equation}\label{abs1002}
\liminf_{k\to\infty} I_{\alpha_k}(\mu_k) \geq \liminf_{k\to\infty} I_{\alpha_k}(\mu_{\alpha_k})= J_\ast(\mu_\ast).
\end{equation}
Since both $\mu_\ast$ and $\nu_\ast$ are minimisers of $J_\ast$, \eqref{abs1001} and \eqref{abs1002} give a contradiction.
On measures with compact support the representation \eqref{Jast-rep} holds
and the right-hand side of \eqref{Jast-rep} is strictly convex as a function of $\mu$. We, thus, conclude that
$\mu_\ast$ is the only minimiser of~$J_\ast$.
\end{proof}

\end{section}

\bigskip

\noindent \textbf{Acknowledgements.} 
JAC acknowledges support by the EPSRC Grant EP/P031587/1. 
JM and JV are supported by MDM-2014-044 (MICINN, Spain),
2017-SGR-395 (Generalitat de Catalunya),
and MTM2016-75390 (Mineco). 
MGM acknowledges support by the Universit\`a di Pavia
through the 2017 Blue Sky Research Project ``Plasticity at different scales: micro to macro"
and by GNAMPA--INdAM.
LR is partly supported by GNAMPA--INdAM through Projects 2018 and 2019.
LS acknowledges support by the EPSRC Grant EP/N035631/1.

\end{document}